\documentclass[reqno]{amsart}
\usepackage{amsmath, amsthm, amssymb, amstext}

\usepackage{hyperref,xcolor}
\hypersetup{pdfborder={0 0 0},colorlinks}
\usepackage{enumitem}
\allowdisplaybreaks
\usepackage{graphicx}
\usepackage{caption}
\usepackage{mathtools}
\usepackage{tikz}
\usepackage{esint}
\usepackage{cancel}
\usepackage{lineno}
\newtheorem{theorem}{Theorem}
\newtheorem{remark}[theorem]{Remark}
\newtheorem{lemma}[theorem]{Lemma}
\newtheorem{proposition}[theorem]{Proposition}

\newtheorem{definition}[theorem]{Definition}

\DeclareMathOperator*{\ess}{\text{ess}}
\let\sp\undefined
\DeclareMathOperator*{\sp}{\text{span}}

\newcommand{\N}{\mathbb{N}}
\newcommand{\R}{\mathbb{R}}
\newcommand{\h}{\mathbb{H}}

\newcommand{\eps}{\varepsilon}
\newcommand{\ph}{\varphi}
\newcommand{\into}{\int_{\Omega}}

\newcommand{\intr}{\iint_{\R^N\times\R^N}} 
\newcommand{\inta}{\iint\limits_{\substack{x,y\in\R^N \\ \abs{x-y}\leq1}}}
\newcommand{\intb}{\iint\limits_{\substack{x,y\in\R^N \\ \abs{x-y} \geq 1}}}

\renewcommand{\l}{\left}
\renewcommand{\r}{\right}

\def\abs#1{\left|{#1}\right|}

\numberwithin{theorem}{section}
\numberwithin{equation}{section}

\renewcommand{\thetheorem}{\arabic{section}.\arabic{theorem}}

\title[Weighted eigenvalue problem]{Spectral Properties of the Logarithmic Laplacian with Indefinite Weights}

\author[R. Arora]{Rakesh Arora} 
\address[R. Arora]{Department of Mathematical Sciences, Indian Institute of Technology Varanasi (IIT-BHU), Uttar Pradesh 221005, India}
\email{rakesh.mat@iitbhu.ac.in}

\author[T. Mukherjee]{Tuhina Mukherjee}
\address[T. Mukherjee]{Department of Mathematics, Indian Institute of Technology Jodhpur, Rajasthan 342030, India}
\email{tuhina@iitj.ac.in}

\author[A. Vaishnavi]{Arshi Vaishnavi} 
\address[A. Vaishnavi]{Department of Mathematical Sciences, Indian Institute of Technology Varanasi (IIT-BHU), Uttar Pradesh 221005, India}
\email{arshiv1998@gmail.com}

\begin{document}
\begin{abstract}
In this paper, we investigate a weighted eigenvalue problem driven by the Logarithmic Laplacian with indefinite weights. We prove the existence of an unbounded sequence of Lusternik-Schnirelman eigenvalues and show that the first eigenvalue is simple, with the associated eigenfunction having constant sign in the domain. In contrast, eigenfunctions corresponding to higher eigenvalues necessarily change sign. We further establish a nodal domain type inequality relating the higher eigenvalues to the measure of the positive and negative parts of the corresponding eigenfunctions, which is of independent interest. As an application, we prove that the first eigenvalue is isolated. In addition, we obtain alternative variational characterizations of the first and second eigenvalues and establish monotonicity properties of the eigenvalues with respect to both the weight function and the domain.
\end{abstract}
\maketitle

\vspace{0.5em}


\vspace{0.5em}
\tableofcontents

\section{Introduction}
Nonlocal operators have become central in the analysis of phenomena involving long-range interactions. In contrast to classical local operators, their action depends on values of a function over the entire domain, making them particularly suitable for modeling anomalous diffusion, phase transitions and L\'evy-type stochastic processes. This intrinsic nonlocality leads to rich mathematical structures but also introduces substantial analytical challenges, requiring the development of generalized functional frameworks and new variational techniques.

A fundamental example is the fractional Laplacian $(-\Delta)^s$, which arises naturally in probability theory, harmonic analysis and partial differential equations. For $s \in (0,1)$, it is defined by (see \cite[Section 3]{valdinoci})
\[
(-\Delta)^su(x) = c(N,s)\, \mathrm{P.V.} \int_{\R^N} \frac{u(x)-u(y)}{|x-y|^{N+2s}} ~dy,
\]
where $c(N,s):=\frac{2^{2s}\pi^{\frac{-N}{2}}s \Gamma(\frac{N+2s}{2})}{\Gamma(1-s)}$ is a normalization constant.

More recently, attention has turned to operators of logarithmic type, which can be interpreted as limiting cases of fractional operators. In particular, the Logarithmic Laplacian $L_\Delta$ arises as the first-order expansion of $(-\Delta)^s$ as $s \to 0^+$
\[
(-\Delta)^s u(x) = u(x) + s L_\Delta u(x) + o(s) \quad \text{in } L^p(\R^N), \; 1 < p \leq \infty.
\]
For $u \in C_c^2(\R^N)$, it admits the integral representation (see \cite{Chen-Weth})
\[
L_\Delta u(x) = c_N \int_{\mathcal{B}_1(x)} \frac{u(x)-u(y)}{|x-y|^N} ~dy 
- c_N \int_{\R^N \setminus \mathcal{B}_1(x)} \frac{u(y)}{|x-y|^N} ~dy + \rho_N u(x),
\]
where $c_N := \pi^{-N/2}\Gamma\!\left(\frac{N}{2}\right)$, $\rho_N := 2\ln 2 + \Psi\left(\frac{N}{2}\right) - \gamma$, $\Psi:=\frac{\Gamma'}{\Gamma}$ is the Digamma function and $\gamma=-\Gamma'(1)$ is the Euler-Mascheroni constant. The operator was systematically introduced in \cite{Chen-Weth} and its analytical properties have since been investigated in a growing body of literature, see, for instance, \cite{Arora-Giacomoni-Hajaiej-Vaishnavi, Arora-Mukherjee, Lara-Saldana,  Chen-Hauer-Weth, Chen-Veron, Chen-Veron-1, Fernandez-Saldana, Santamaria-Rios-Saldana}.
 
Eigenvalue problems for both local and nonlocal operators are fundamental in understanding qualitative properties of solutions. Eigenvalues reflect important analytical properties of the underlying differential operator and the domain, including positivity, boundary effects and qualitative behavior of solutions (see, for instance, \cite{Brezis-Nirenberg, Crandall-Rabinowitz}). They play a key role in variational methods and spectral theory (see \cite{Gilbarg-Trudinger, Reed-Simon}). For these reasons, the study of eigenvalue problems has become an essential part of the theory of partial differential equations.

On the other hand, weighted eigenvalue problems have been widely investigated for both local and nonlocal operators, including the classical Laplacian \cite{Allegretto, Anoop-Lucia-Ramaswamy, Brown, Szulkin-Willem}, the $p$-Laplacian \cite{Cuesta} and the fractional $p$-Laplacian \cite{Iannizzotto}. From a theoretical perspective, the presence of a weight function enriches the spectral structure of the problem and leads to new analytical challenges. The weight influences the distribution of eigenvalues, the behavior of eigenfunctions and the associated variational framework (see, for instance,  \cite{Brown-Lin, Hess-Kato}). Such problems also play a central role in the study of nonlinear and nonlocal equations, since the spectral properties of the underlying operator strongly influence the existence, multiplicity and qualitative behavior of solutions. 

For the fractional Laplacian, the eigenvalue problem  
\begin{equation*}
(-\Delta)^s u = \lambda g(x) u \quad \text{in } \Omega, 
\qquad u = 0 \quad \text{in } \mathbb{R}^N \setminus \Omega
\end{equation*}
has been extensively studied in recent years. In the case \(g \equiv 1\), the problem was first introduced by Lindgren and Lindqvist in \cite{Lindgren-Lindqvist} and subsequently investigated by several authors; see, for instance, \cite{Brasco-Parini, Franzina-Palatucci, Iannizzotto-Squassina}. Among the main developments, the existence of an infinite sequence of eigenvalues diverging to infinity was established in \cite{Iannizzotto-Squassina}. Further qualitative properties have also been obtained, including the simplicity of the first eigenvalue and a variational characterization of the second eigenvalue \cite{Brasco-Parini}, nodal properties of higher eigenfunctions \cite{Lindgren-Lindqvist} and Weyl-type asymptotic estimates \cite{Iannizzotto-Squassina}.  

The weighted case was later considered in \cite{Iannizzotto}, under the assumption that \(g \in L^{\frac{N}{2s}}(\Omega)\) with \(g^+ \not\equiv 0\). Using the \(\mathbb{Z}_2\)-cohomological index of Fadell and Rabinowitz, the authors proved the existence of an unbounded sequence of minimax eigenvalues, with the first eigenvalue being simple and isolated. In addition, they showed that higher eigenfunctions are nodal and established monotonicity properties of the first two eigenvalues with respect to the weight \(g\) and the domain \(\Omega\). We also refer to \cite{Asso et al, Ho-Perera-Sim-Squassina-2017, Ho-Sim-2019} for further results on related weighted nonlocal eigenvalue problems. In contrast to the fractional Laplacian, the spectral theory of the Logarithmic Laplacian is still comparatively less developed. 

For the Logarithmic Laplacian, the eigenvalue problem
\begin{equation*}
L_\Delta u = \lambda u \quad \text{in } \Omega, 
\qquad u = 0 \quad \text{in } \mathbb{R}^N \setminus \Omega
\end{equation*}
has attracted increasing attention in recent years. In \cite{Chen-Weth}, it was shown that, on bounded domains, the operator \(L_\Delta\) possesses a discrete spectrum together with a complete orthonormal basis of eigenfunctions. Moreover, the asymptotic expansion
\[
\lambda_{1,s} = 1 + s\,\lambda_{1,L} + o(s)
\quad \text{as } s \to 0^+,
\]
where \(\lambda_{1,s}\) and \(\lambda_{1,L}\) denote the first eigenvalues of the fractional and Logarithmic Laplacians, respectively, establishes a connection between the corresponding spectral theories. Subsequent developments include the convergence of eigenpairs as \(s \to 0^+\) \cite{Feulefack-Jarohs-Weth}, as well as spectral estimates such as Weyl-type asymptotics and eigenvalue bounds \cite{Chen-Veron, Laptev-Weth}. A fundamental feature distinguishing the Logarithmic Laplacian from the fractional Laplacian is, the first eigenvalue of the Logarithmic Laplacian may become non-positive and failure of maximum principle on sufficiently large domains (see \cite{Chen-Weth, Chen-Veron}).

Motivated by the analytical challenges associated with the Logarithmic Laplacian, in this paper we investigate the weighted eigenvalue problem
\begin{equation}\label{equ}
    L_\Delta u = \lambda \omega(x) u \quad \text{in } \Omega, 
    \qquad u = 0 \quad \text{in } \R^N \setminus \Omega,
\end{equation}
where \(\Omega \subseteq \R^N\) is a bounded domain and \(\omega : \Omega \to \mathbb{R}\) is a sign-changing weight function.

The aim of the present work is to develop a variational and spectral framework for problem \eqref{equ} while addressing the new challenges arising from nonlocal nature of the operator $L_\Delta$ and from the presence of the sign-changing weights.

The nonlocal nature of the operator \(L_\Delta\), together with the indefinite nature of the associated quadratic form (see below \eqref{quadratic-form}), gives rise to several analytical and technical difficulties that prevent a direct adaptation of the methods developed for the classical and fractional Laplacians. In particular, the bilinear form may change sign due to the presence of a nonlocal contribution involving interactions outside the unit ball, for which no definite sign information is available. As a consequence, the associated energy functional may fail to be bounded from below, creating substantial obstacles in the variational analysis of the minimization problem related to the Logarithmic Laplacian. 

Another challenging problem arises when considering eigenvalue problems with sign-changing weights. In contrast to the unweighted setting considered in \cite{Chen-Weth}, the associated energy functional may no longer be bounded from below, even under weighted normalization constraints. As a consequence, minimizing sequences may escape to infinity, leading to a loss of compactness, thereby obstructing the use of standard variational arguments. These difficulties reveal the necessity of imposing appropriate structural assumptions on the weight in order to establish a suitable variational framework. We refer the reader to the next section for a detailed presentation on the weight function assumptions, the principal difficulties arising in the analysis of the weighted eigenvalue problem \eqref{equ} and comparison of our findings with the existing literature.

\textbf{Outline of the paper:} The paper is organized as follows. In Section \ref{main-res}, we introduce the necessary definitions, notations and preliminary material required for the subsequent analysis and state the main results. In Section \ref{sec-existence}, we prove the existence of a sequence of eigenvalues of \eqref{equ} diverging to infinity and the principal eigenvalue by means of the Rayleigh quotient coincides with Lusternik-Schnirelman eigenvalue. Section \ref{sign properties} is devoted to the sign-changing nature of eigenfunctions, together with a nodal domain type estimate for eigenvalues other than principal eigenvalue. In Section \ref{behav-char}, we study the simplicity and isolation of the first eigenvalue, provide an alternative variational characterization of the second eigenvalue and establish the monotonicity of eigenvalues $\lambda_k$, $k \in \mathbb{N}$ and strict monotonicity of $\lambda_1$ and $\lambda_2$ with respect to both the domain $\Omega$ and the bounded weight function $\omega$. The paper ends with an Appendix \ref{Appendix} which contains the proof of Theorem \ref{Hardy-inequ}.

\section{Preliminaries and Main results}
\label{main-res}
\subsection{Function spaces and preliminaries}
Let $\Omega \subseteq \R^N$ be a bounded domain with Lipschitz boundary. For $q\in [1,\infty]$, we denote by $L^q(\Omega)$ the standard Lebesgue space with the norm
\[
\|u\|_{L^q(\Omega)}:= \left(\into |u|^q ~dx\right)^\frac{1}{q}\ \text{for} \ 1 \leq q < \infty \quad \text{and} \quad \|u\|_{L^\infty(\Omega)}:= \ess\sup_\Omega |u|.
 \]
 Define
\[
k, j: \mathbb{R}^N \setminus \{0\} \to \mathbb{R} \quad \text{as} \quad k(z) =  \frac{{\bf 1}_{B_1(z)}}{|z|^N} \quad \text{and} \quad j(z) =  \frac{{\bf 1}_{\mathbb{R}^N \setminus B_1(z)}}{|z|^N}.
\]
For the problem \eqref{equ}, the natural solution space $\h(\Omega)$ is defined as (see \cite{Chen-Weth})
\[
\begin{split}
\h(\Omega) = \bigg\{u\in L^2(\Omega): & \ u=0 \  \text{in} \ \R^N \setminus \Omega \ \text{and}\\
& \intr \abs{u(x)-u(y)}^2 k(x-y) ~dx ~dy < +\infty\bigg\}.
\end{split}
\]
The inner product and norm on $\h(\Omega)$ are given by
\[
\mathcal{E}(u,v) = \frac{c_N}{2}\intr (u(x)-u(y))(v(x)-v(y)) k(x-y) ~dx ~dy\]
and  
\[ \|u\|_{\h(\Omega)}= (\mathcal{E}(u,u))^\frac{1}{2}.\]
The quadratic form associated to $L_\Delta$ is defined as
\begin{equation}\label{quadratic-form}
    \mathcal{E}_L(u,v) = \mathcal{E}(u,v)-c_N \intr u(x)v(y) j(x-y) ~dx ~dy + \rho_N\int_{\R^N} uv ~dx.
\end{equation}
\begin{definition}
\label{weak-formulatn}
A function $u \in \h(\Omega)$ is said to be an eigenfunction (or equivalently a weak solution to \eqref{equ}) corresponding to an eigenvalue $\lambda\in \R$ of \eqref{equ}, if
\[\mathcal{E}_L(u,v)=\lambda\into \omega(x) u v ~dx \quad \text{for all} \ v\in \h(\Omega).\]
\end{definition}
\begin{theorem}\cite[Theorem 2.1]{Correa-DePablo}
    The space $\h(\Omega)$ is compactly embedded in $L^2(\Omega)$.
\end{theorem}
\begin{lemma}
\label{est:lower-order-term}\cite[Lemma 3.4]{Santamaria-Saldana} Let $u,v\in L^2(\Omega)$ be such that $v=u=0$
in $\R^N\setminus\Omega$, then
\[
\intb \frac{|u(x)v(y)|}
{|x-y|^N} ~dx~dy
\leq \|u\|_{L^1(\Omega)}\|v\|_{L^1(\Omega)}\leq|\Omega|\|u\|_{L^2(\Omega)}\|v\|_{L^2(\Omega)}.\]
\end{lemma}
 The embedding result established in \cite[Theorem 2.1]{Correa-DePablo} are further improved to optimal continuous and compact embeddings in \cite[Theorem 2.2]{Arora-Giacomoni-Vaishnavi}. To this end, we recall some preliminary definitions and notations.

\begin{definition}\cite[Definition 2.1.3]{Harjulehto-Hasto}\label{def:phi-function}
	A function $\zeta \colon [0,+\infty) \to [0,+\infty]$ is said to be a $\Phi$- prefunction, if $\zeta$ is increasing and satisfies 
    \[
    \zeta(0)=0, \quad \lim_{t\to 0^+} \zeta(t) = 0, \quad \text{and} \quad \lim_{t \to +\infty} \zeta(t) = +\infty.
    \]
Moreover, $\zeta$ is said to be a convex $\Phi$-function (denoted by $\zeta \in \Phi_c$) if $\zeta$ is left continuous and convex.
    \end{definition}

\begin{definition}\cite[Definition 2.4.1]{Harjulehto-Hasto}
    \label{young-conj-def}
   Let $\zeta : [0,\infty) \to [0,\infty]$. We denote by $\zeta^\ast$, the Young conjugate function 
of $\zeta$, which is defined as
\[
\zeta^\ast(s) := \sup_{t \geq 0} \{ ts - \zeta(t) \} \quad \text{for} \ s \geq 0.
\]
\end{definition}

Let $M(\Omega)$ be the set of all real valued Lebesgue measurable functions defined on $\Omega$. 
Let $\zeta \colon [0,+\infty) \to [0,+\infty]$ be a convex $\Phi$-function. The modular associated to $\zeta$ is defined as
	\[
		\varrho_\zeta (f) = \into \zeta(f (\abs{x})) ~dx.
	\]
	The set
	\begin{align*}
		L^\zeta (\Omega) := \{ f \in M(\Omega)\,:\, \varrho_\zeta (\lambda f) < +\infty \text{ for some } \lambda > 0 \}
	\end{align*}
	equipped with the Luxemburg norm
	\begin{equation}
    \label{orlicz-norm-def}
		\|f\|_{L^\zeta(\Omega)} = \inf \bigg\{ \lambda > 0 \,:\, \varrho_\zeta \l( \frac{f}{\lambda} \r)  \leq 1 \bigg\}
\end{equation}
is a Banach space (see \cite[Theorem 3.3.7]{Harjulehto-Hasto}). 
Moreover, the left inverse  $\zeta^{-1}_{\text{left}}:[0,\infty] \to [0,\infty]$ of $\zeta$ (see \cite[Definition 2.3.1]{Harjulehto-Hasto}) is given by
      \[
      \begin{split}
\zeta^{-1}_{\text{left}}(x)&=\inf\{\tau\geq 0: \zeta(\tau)\geq x\}.
      \end{split}
      \]

\begin{theorem}\cite[Theorem 2.2 (ii)]{Arora-Giacomoni-Vaishnavi}\label{cpct-embed} 
Let $\Omega\subseteq \R^N$ be a bounded domain and $\ph, \eta$ be convex $\Phi$-functions such that $\ph(t):= t^2 \ln(e+t)$ for all $t \geq 0$ 
and $\eta$ is defined as in \eqref{eta}.
Then, the embedding $\h(\Omega)\hookrightarrow L^\ph(\Omega)$ is continuous and $\h(\Omega)\hookrightarrow L^\eta(\Omega)$ is compact.
\end{theorem}
     
Next, we recall some known definitions and lemmas.
Let $E$ be a real Banach space. 
\begin{definition}\cite[Page 1]{Rabinowitz}
Let $I:E\to \R$ be a $C^1$ functional. A critical point $u$ of  $I$ is a point at which $I'(u) =0$. The value of $I$ at $u$ is then called a {\it critical value} of $I$. 
A number $c$ is called a {\it regular value} of $I$ if 
\[
I'(u) \neq 0 \quad \text{for all } u \in E \text{ such that } I(u) = c.
\]
\end{definition}
Now, consider a $C^1$ functional $g$ such that $1$ is not a critical value of $g$ and $f$ a $C^1$ functional defined in a neighborhood of $S:= \{ z \in E \mid g(z) = 1 \}$. For $\beta \in \R$, define
$
S_\beta = \{ z \in E :g(z)=1+\beta \}.
$
Let $a$ be a regular value of $f|_S$.
\begin{definition}
\cite[Definition 2.3]{Bonnet}
A family $F_\alpha$ of subsets of $S$ is said to be {\it admissible} for $f$ at $a \in \R$ if $f$ is defined on $S_\beta$, for all $\beta \in [0,\alpha]$ and there exists $\mu,\delta,\epsilon_1>0$ such that
\[|f(x)-a|\leq\epsilon_1 \ \Longrightarrow \quad \|f'\big|_{S_\beta}(x)\|>\delta \ \text{and } \|g'(x)\|>\mu\]
for all $\beta \in [0,\alpha]$ and $x \in S_\beta.$
\end{definition}
\begin{lemma}\label{def-lem}\cite[Theorem 2.5]{Bonnet}

Let $f$ be a $C^1$ functional defined on a neighborhood of $S$ and $a$ be a regular value of the restriction $f|_S$ and $\alpha$ such that $F_\alpha$ is admissible for $f$ at the level $a$. Then, there exists $\eps > 0$ such that for all $0 < \epsilon_1 < \eps$, there exists a homeomorphism $\eta$ from $S$ onto $S$ satisfying
\begin{itemize}
    \item[\textnormal{(i)}] $\eta(z) = z$ if $f(z) \notin [a-\epsilon_1,\, a+\epsilon_1]$
    
    \item[\textnormal{(ii)}] $f(\eta(z)) \leq f(z)$ for all $z \in S$
    
    \item[\textnormal{(iii)}] $f(\eta(z)) \leq a - \epsilon_1$ for all $z \in S$ such that $f(z) < a + \epsilon_1$
    
    \item[\textnormal{(iv)}] If $S$ is symmetric ({\it i.e.,} $S = -S$) and $f$ is even, then $\eta$ is odd.
\end{itemize}
\end{lemma}

\begin{definition}
 \label{genus-def}   \cite[Page 45]{Rabinowitz}
Let 
\[V:=\{A\subseteq E\setminus\{0\}:\, A \text{ is closed in } E \text{ and } A=-A\}.\]
For $A\in V$, define the (Kranoselskii) genus of $A$ to be 
\[\gamma(A)=\min\{n \in \N:\, \text{there is an odd map } \nu \in C(A, \R^n \setminus \{0\}\}.\]
When there does not exist a finite such $n$, set $\gamma(A)=+\infty$. Finally, set $\gamma(\emptyset) = 0$.
\end{definition}

\begin{lemma}
\label{genus-prop}
 \cite[Proposition 7.5]{Rabinowitz}
  Let $A, B \in V$. Then,
  \begin{enumerate}
\item[\textnormal{(i)}] If $x\neq 0,$ $\gamma(\{x\}\cup\{-x\})=1$
\item[\textnormal{(ii)}] If there exists an odd map $f \in C(A,B)$ then $\gamma(A)\leq \gamma(B)$
\item[\textnormal{(iii)}] Let $Z$ be a subspace of $E$ with codimension $k$ and $\gamma(A) > k$. Then
\[
A \cap Z \neq \emptyset.
\] 
\end{enumerate}
\end{lemma}
\subsection{Main assumptions and results}

We assume the weight function $\omega:\Omega \to \R$ in \eqref{equ} satisfies the following conditions:
\begin{enumerate}
[label=\textnormal{($\omega_0$)}, ref=\textnormal{$\omega_0$}]
\item  \label{assump w0}  $\omega^+ := \max\{\omega,0\} \not \equiv 0.$
\end{enumerate}

\begin{enumerate}
[label=\textnormal{($\omega_1$)}, ref=\textnormal{$\omega_1$}]
\item  \label{assump w1} 
$\omega \in  L^{\psi^\ast}(\Omega)$ with
\begin{equation}\label{psi}
    \psi(t):=\eta(\sqrt{t}) \quad \text{for all} \ t \geq 0,
\end{equation}
where $\eta:[0,\infty) \to [0,\infty)$ is a convex $\Phi$-function (see Definition \ref{def:phi-function}) such that
\begin{equation}\label{eta} 
    {\lim\limits_{t\to 0^+} \frac{\eta(t)}{t} \text{ exists in } \mathbb{R}}, \quad \text{and} \quad \lim\limits_{t \to \infty} \frac{\eta(t)}{t^2 \ln(e+t)} = 0.
\end{equation} 
Here, $\psi^\ast$ denotes the Young conjugate of $\psi$ (see Definition \ref{young-conj-def}).
\end{enumerate}

\begin{enumerate}
[label=\textnormal{($\omega_2$)}, ref=\textnormal{$\omega_2$}]
\item  \label{assump w2} 
There exist $\lambda_0 \in \R \setminus \{0\}$ and $\alpha\in \R^+$ such that
\[\ln\l(\frac{1}{|x|^2}\r)-\lambda_0\,\omega(x)\geq \alpha-\ln 4-2\Psi(\frac{N}{4}) \quad \text{a.e. $x$} \in \Omega.\]
\end{enumerate}

\begin{remark}
Some examples of weight function $\omega$ satisfying assumptions \eqref{assump w0}-\eqref{assump w2} are
\begin{enumerate}
    \item  [\textnormal{(i)}] Indefinite weight function: $\omega \in L^\infty(\Omega)$ such that $\omega^+\not\equiv 0.$ Since $\Omega$ is bounded domain, there exists a $\delta>0$ such that $|x| \leq \delta$ for all $x \in \overline{\Omega}$. Then, for any $\alpha \in \mathbb{R}^+$ and $\lambda_0 \leq \left(\ln 4 + 2\Psi(\frac{N}{4}) - 2 \ln \delta - \alpha\right)/\|\omega\|_{L^\infty(\Omega)}$, \eqref{assump w0}-\eqref{assump w2} holds true.
    \item  [\textnormal{(ii)}] Singular weight function: For any $\beta \in (0,1)$,  $\gamma < \frac{N}{2}$ and $g \in L^\infty(\Omega)$ satisfying \eqref{assump w2}, we define
    \[
\omega(x)=\begin{cases} \ln^\beta \l(\frac{1}{|x|^{\gamma}}\r) & x\in \mathcal{B}_1(0)\\
g(x) &x \in \Omega\setminus \mathcal{B}_1(0).
\end{cases}
\]
and $\psi(t) = \eta(\sqrt{t}) = t (\ln(e+\sqrt{t}))^{\beta}.$ It is easy to check that there exist $\lambda_0 \in \R \setminus \{0\}$ and $\alpha\in \R^+$ such that $\omega$ satisfies \eqref{assump w2}. In view of Definition \ref{young-conj-def}, by straightforward computations, there exists a $M \gg 1$ such that
\[
\psi^\ast (s) \sim s^{\frac{\beta-1}{\beta}} e^{2s^\frac{1}{\beta}} \quad \text{for} \quad s \gg M.
\]
Then, 
\[
\begin{split}
    \int_{\mathcal{B}_{M_\beta}(0)} \psi^\ast(\omega(x)) ~dx 
    & \leq C \int_{\mathcal{B}_{M_\beta}(0)}  \l(\ln \frac{1}{|x|^{\gamma}}\r)^{\beta-1} \frac{1}{|x|^{2\gamma}} ~dx < +\infty
\end{split}
\]
where $M_\beta := \exp(-M^\beta).$
\end{enumerate}
\end{remark}

For $u\in\h(\Omega)$, we define the functionals $I, J: \h(\Omega) \to \R$ such that
\[
I(u):=\mathcal{E}_L(u,u) \quad \text{and} \quad J(u)=\into \omega(x) u^2 ~dx.
\]

Note that $0$ is the only critical point of the functional $J$. Hence, $1$ is a regular value of $J$. Thus, we define the constraint manifold $M$ as follows
\[
M:=\{u \in \h(\Omega)\,:\, J(u)=1\}.
\]
For all $k \in \mathbb{N}$, set
\[F_k:=\{A\subseteq M\,:\, A=-A \ \text{is closed }, \ \gamma(A)\geq k\}\]
where $\gamma(\cdot)$ represent the Kranoselskii's genus defined in Definition \ref{genus-def}. We define a sequence of variational (Lusternik-Schnirelman) eigenvalues of \eqref{equ} as
    \begin{equation}\label{lambda}
        \lambda_k:=\inf_{A\in F_k} \sup_{u\in A} I(u) \quad \text{for all} \quad k \in \mathbb{N}.
    \end{equation} 
The first result corresponds to the existence of a unbounded sequence of eigenvalues and the principal eigenvalue by means of the Rayleigh quotient coincides with Lusternik-Schnirelman eigenvalue.
\begin{theorem}
    Assume \eqref{assump w0}-\eqref{assump w2}  hold true. The problem \eqref{equ} admits a sequence of eigenvalues 
    \[\lambda_0  \leq \lambda_1\leq\lambda_2\leq \cdots \leq \lambda_k\leq \cdots\] that diverges to $+\infty,$ where $\lambda_0$ is the constant in \eqref{assump w2}. Moreover, the infimum
    \begin{equation*}
     \lambda_1^R :=\inf_{u\in M}\ \mathcal{E}_L(u,u),
\end{equation*} is achieved and coincides with $\lambda_1.$
\end{theorem}

The existence of the sequence of eigenvalues $(\lambda_k)_{k \in \mathbb{N}}$ stated above is established by means of the Lusternik-Schnirelman variational principle under assumptions \eqref{assump w0}-\eqref{assump w2}. In particular, assumption \eqref{assump w0} ensures that the associated constrained manifold $M$ is nonempty, thereby providing a suitable variational framework for the eigenvalue problem.

In the context of the fractional Laplacian, the corresponding weighted eigenvalue problem was studied in \cite{Iannizzotto} under the requirement that the weight function belongs to $L^{\frac{N}{2s}}(\Omega)$. The exponent $\frac{N}{2s}$ appears naturally as the conjugate exponent of $\frac{2_s^\ast}{2}$, where $2_s^\ast$ denotes the fractional critical Sobolev exponent. Consequently, in the limiting regime $s \to 0^+$, the natural admissible class of weights reduces to $L^\infty(\Omega)$. However, due to the optimal continuous and compact embedding results obtained in \cite[Theorem 2.2]{Arora-Giacomoni-Vaishnavi}, we are able to work with a substantially broader class of admissible weights, namely the Orlicz-type spaces described in assumption \eqref{assump w1}. This setting not only remains consistent with the weighted framework arising in the fractional Laplacian case, but also significantly extends it to a larger class of admissible weights. Moreover, the choice of the Orlicz space in \eqref{assump w1} is crucial in deriving a key inequality established in Lemma \ref{lem:norm-est}, which in turn plays an essential role in obtaining convergence estimates for the associated weighted integrals. However, the techniques used in the convergence estimates \cite[Lemma 2.1]{Iannizzotto} are not applicable in our framework, owing to the lack of density of smooth functions in the Orlicz-type space defined in \eqref{assump w1}.

The assumption \eqref{assump w2} is inspired by Pitt-type inequalities (see \cite[Theorem 1]{Beckner}) and plays a role analogous to that of Hardy-type inequalities within the present framework. In particular, it allows us to recover an appropriate variational structure even in the presence of singular and indefinite weights. This condition ensures that the associated energy functional is bounded from below on the constrained manifold $M$ and further yields the boundedness of Palais-Smale sequences required for the application of the Lusternik-Schnirelman principle. Finally, in order to prove that the spectrum diverges to infinity, we exploit the separability of the underlying energy space together with the existence of a biorthogonal double sequence in $\h(\Omega)\times (\h(\Omega))^\ast$.

Next, we prove the sign properties of the eigenfunctions and nodal domain inequality. 
\begin{theorem}
   Assume \eqref{assump w0}-\eqref{assump w2}  hold true. Then the following hold:
    \begin{enumerate}
        \item[\textnormal{(i)}] Eigenfunction corresponding to the first eigenvalue $\lambda_1$ has a constant sign in the domain $\Omega$.
        \item[\textnormal{(ii)}] Let $\omega \in L^\infty(\Omega).$ Then, any eigenfunction corresponding to eigenvalue $\lambda\neq \lambda_1$ changes sign in the domain $\Omega$.
        \item[\textnormal{(iii)}] Let $\psi$ satisfies \eqref{psi}-\eqref{eta} and $\vartheta_1, \vartheta_2$ be Orlicz functions in $\Phi_c$ such that \[(\vartheta_{1})^{-1}_{\text{left}}(x) (\vartheta_2)^{-1}_{\text{left}}(x)\leq (\psi^\ast)^{-1}_{\text{left}}(x)\]
and $\omega\in L^{\psi^\ast}(\Omega)\cap L^{\vartheta_1}(\Omega)$. Let $v$ be an eigenfunction associated to an eigenvalue $\lambda \neq \lambda_1$. Then, we have 
   \[|\lambda|\geq \frac{\vartheta_2^{-1}(\frac{1}{|\Omega^\pm_v|})\l(1- \max\{(c_N|\Omega^\pm_v|-\rho_N),0\}S_2\r)}{2S_1\|\omega\|_{L^{\vartheta_1}(\Omega)}},\]
      where $\Omega^\pm:=\{x\in \Omega:u(x)\gtrless 0\}$ and $S_1, S_2$ are the best constants of the embeddings $\h(\Omega)\hookrightarrow L^\ph(\Omega)$ and $\h(\Omega) \hookrightarrow L^2(\Omega),$ respectively.
    \end{enumerate}
\end{theorem}
The proof relies on the geodesic convexity of the energy functional and the regularity properties of the eigenfunction which follows from the additional regularity on the weight function. Similar arguments along with the truncation techniques are applied for the fractional Laplacian problem in \cite{Franzina-Palatucci, Iannizzotto}, however in our case, the additional nonlocal term with no sign information creates new difficulties and prevents us to use the truncation arguments, hence additional regularity of weight is needed. The assumption \eqref{assump w1} also helps in deriving the nodal domain inequality relating the eigenvalues $|\lambda_k|$ for $k >1$ and the measure of positive and negative parts of corresponding eigenfunctions. Due to the contrasting feature of the Logarithmic Laplacian, the eigenvalue may become non-positive, the nodal domain inequality is established for $|\lambda_k|$ instead of $\lambda_k.$ 

Next, we derive the qualitative and monotonicity properties of the eigenvalues. 
\begin{theorem}
  Assume \eqref{assump w0}-\eqref{assump w2}  hold true. Then the following  hold true:
    \begin{enumerate}
        \item[\textnormal{(i)}] The first eigenvalue $\lambda_1$  is simple. Moreover, if $\omega \in L^\infty(\Omega)$ then $\lambda_1$ is also isolated.
        \item[\textnormal{(ii)}] $\lambda_2=\min\{\lambda>\lambda_1:\lambda \text{ is an eigenvalue of } \eqref{equ}\}.$
        \item[\textnormal{(iv)}] Let $\omega\in L^\infty(\Omega).$ Then, 
            \[\lambda_2:=\inf_{A\in F_2} \sup_{u \in A} \mathcal{E}_L(u,u)=\inf_{f\in \mathcal{C}} \max_{\delta\in \mathbb{S}^1} \mathcal{E}_L(f(\delta),f(\delta)),\]
    where $\mathcal{C}$ denotes the set of all odd continuous maps from $\mathbb{S}^1$ to the constrained manifold $\{u \in \h(\Omega)\,:\, \int_\Omega \omega(x) u^2 ~dx=1\}.$
    \end{enumerate}
\end{theorem}
\begin{theorem}
Assume \eqref{assump w0}-\eqref{assump w2}  hold true. Then the following eigenvalue monotonicity properties hold:
    \begin{enumerate}
        \item[\textnormal{(i)}] Let $\omega_1(x)\leq \omega_2(x)$ a.e. in $\Omega$.
Then for every $k\geq1$, we have
\[
\lambda_k(\omega_2)\leq \lambda_k(\omega_1).
\]
\item[\textnormal{(ii)}] Let $\omega_1, \omega_2 \in L^\infty(\Omega)$. Then,
\[
\lambda_1(\omega_1) > \lambda_1(\omega_2) \quad \text{if} \quad \omega_1 \leq \omega_2 \ \text{in} \ \Omega \ \text{and} \ \omega_1 \not\equiv \omega_2,
\]
and
\[
\lambda_2(\omega_1) > \lambda_2(\omega_2) \quad \text{if} \quad \omega_1 < \omega_2 \ \text{in} \ \Omega.
\]
        \item[\textnormal{(iii)}] Let $\Omega_1 \subsetneq \Omega_2\subseteq \R^N$ be bounded domains. Then,
\[
\lambda_k(\Omega_2) \leq \lambda_k(\Omega_1)
\quad \text{for all } k\geq 1 .
\]
Moreover, if $w \in L^\infty(\Omega)$, then $\lambda_1(\Omega_2) < \lambda_1(\Omega_1).$
    \end{enumerate}
\end{theorem}

\section{Existence of eigenvalues}
\label{sec-existence}
In this section, we first present the Hardy inequality associated with our operator $L_\Delta$, which plays a fundamental role in the analysis of the weighted eigenvalue problem \eqref{equ}. For completeness, a detailed proof of this result-also referred to as Pitt’s inequality (see \cite[Theorem 1]{Beckner})-is provided in Appendix \ref{Appendix}. 

\begin{theorem} 
    \label{Hardy-inequ}
    For any $u \in \h(\Omega)$, it holds that
    \[\int_{\R^N} u^2 \ln\l(\frac{1}{|x|}\r) ~dx+\l(\ln2+\Psi(\frac{N}{4})\r)\int_{\R^N} u^2 ~dx\leq\frac{\mathcal{E}_L(u,u)}{2}.\]
\end{theorem}

Next, we derive a key inequality that plays a central role in the analysis of integrals involving the sign-changing weight function $\omega$. 

\begin{lemma} \label{lem:norm-est}
Let $\eta$ satisfies \eqref{psi}. Then, for any $u \in L^\eta(\Omega)$, the following holds
\[
\|u^2\|_{L^\psi(\Omega)} \leq \|u\|_{L^\eta(\Omega)}^2.
\]
\end{lemma}
\begin{proof}
Set
\[ v:= \frac{u}{\|u\|_{L^\eta(\Omega)}} \quad \text{such that} \quad \|v\|_{L^\eta(\Omega)} =1.
\]
By using the unit ball property of the Orlicz norm (see \cite[Lemma 3.2.3]{Harjulehto-Hasto}), we have
\[
\begin{split}
\int_{\Omega} \eta(v) ~dx \leq 1 \ \Longrightarrow \ \int_{\Omega} \psi(v^2) ~dx \leq 1 \ \Longrightarrow \ \|v^2\|_{L^\psi(\Omega)} \leq 1.\ 
\end{split}
\]
Consequently, 
\[
\|u^2\|_{L^\psi(\Omega)} \leq \|u\|_{L^\eta(\Omega)}^2.
\]
\end{proof}
Next, by combining the above result with the optimal compact embeddings established in \cite[Theorem 2.2]{Arora-Giacomoni-Vaishnavi}, we derive convergence estimates for weighted integrals, which are fundamental to the development of the associated variational framework.
\begin{lemma}
    \label{limit passage}
    Assume \eqref{assump w1} holds true. Let $(u_n)_{n\in \N}\subseteq L^\eta(\Omega)$ be such that $u_n\to u$ in $L^\eta(\Omega).$ Then, the following hold true
    \begin{enumerate}
        \item[\textnormal{(i)}] $\displaystyle \lim\limits_{n \to \infty} \into \omega(x) u_n^2 ~dx= \into \omega(x) u^2 ~dx \quad \text{and} \quad \lim\limits_{ n \to \infty} \into \omega(x) (u_n-u)^2 ~dx= 0.$\\
        \item[\textnormal{(ii)}] $\displaystyle \lim\limits_{n \to \infty} \into \omega(x) u_n v~dx= \into \omega(x) uv ~dx \quad \text{for every } v \in L^\eta(\Omega).$
    \end{enumerate}
\end{lemma}
\begin{proof}
Let $\eps>0.$ By the strong convergence of $(u_n)_{n \in\N}$ in $L^\eta(\Omega)$, there exists a $\delta>0$ such that
\[
\|u_n\|_{L^\eta(A)}^2 < \frac{\eps}{2 \|\omega\|_{L^{\psi^\ast}(\Omega)}}
\]
for any $A \subset \Omega$ and $|A| < \delta.$
By applying H\"older's inequality and using Lemma \ref{lem:norm-est}, the above estimate further implies
\[
    \begin{split}
\int_{A} \omega(x) u_n^2 ~dx \leq 2 \|\omega\|_{L^{\psi^\ast}(A)}\|u_n^2\|_{L^\psi(A)}\leq 2 \|\omega\|_{L^{\psi^\ast}(A)}\|u_n\|_{L^\eta(A)}^2<\eps.
    \end{split}
    \] 
    Hence, $\{\omega u_n^2\}_{n \in \N}$ is uniformly integrable. By Vitali convergence theorem, we get the first claim in (i).  Again, by using H\"older's inequality, Lemma \ref{lem:norm-est} and $u_n \to u$ in $L^\eta(\Omega)$, we have 
    \[
    \begin{split}
       \l| \into \omega(x) (u_n-u)^2 ~dx\r| & \leq 2 \|\omega\|_{L^{\psi^\ast}(\Omega)} \|(u_n-u)^2\|_{L^\psi(\Omega)} \\
       & \leq 2 \|\omega\|_{L^{\psi^\ast}(\Omega)}\|u_n-u\|_{L^\eta(\Omega)}^2 \to 0 \quad \text{as} \ n \to \infty.
    \end{split}
    \] 
    This proves (i). To prove (ii), first note that by H\"older's inequality and Lemma \ref{lem:norm-est}, we have 
    \[
    \begin{split}
        \into \omega(x) (u_n-u) v ~dx & \leq \l(\into \omega(x) v^2 ~dx\r)^{\frac{1}{2}} \l(\into \omega(x) (u_n-u)^2 ~dx\r)^{\frac{1}{2}}\\
        & \leq 2 \|\omega\|_{L^{\psi^\ast}(\Omega)} \|u_n-u\|_{L^\eta(\Omega)} \|v\|_{L^\eta(\Omega)} \to 0 \ \text{as} \ n \to \infty.
    \end{split}
    \]
Hence, the claim in (ii).
\end{proof}

Next, in order to minimize the energy functional $I$ on the constrained manifold $M$ and to apply Lusternik-Schnirelman principle, we show that $I$ is bounded below on $M$ and satisfies the Palais-Smale condition at the min-max levels $\lambda_k$.
\begin{lemma}
  \label{energy bdd below} 
 Assume \eqref{assump w0} and \eqref{assump w2} hold true. Then, the functional $I(\cdot)$ is bounded below by $\lambda_0$ on the set $M$.
\end{lemma}
\begin{proof}
Note that by \eqref{assump w2}, we have
\[\ln\l(\frac{1}{|x|^2}\r)+\ln 4+2\psi(\frac{N}{4})-\lambda_0\,\omega(x)>0 \quad \text{a.e. $x$} \in \Omega.\]
Using this and Theorem \ref{Hardy-inequ}, we get
    \[
    \begin{split}
I(u) &\geq 2\int_{\R^N} u^2 \ln\l(\frac{1}{|x|}\r) ~dx+2\l(\ln2+\psi(\frac{N}{4})\r)\int_{\R^N} u^2 ~dx\\&>\lambda_0\into \omega(x) u^2 ~dx=\lambda_0.
    \end{split}
    \]
    This establishes the claim.
\end{proof}
\begin{remark}
    Using the definition of $\mathcal{E}_L(\cdot,\cdot)$ and Lemma \ref{est:lower-order-term}, it is easy to see that
 \begin{equation}
\label{E-L-bd-below-above}
\mathcal{E}(u,u)-(c_N|\Omega|-\rho_N)\|u\|_{L^2(\Omega)}^2 \leq \mathcal{E}_L(u,u) \leq \mathcal{E}(u,u) + (c_N|\Omega| + \rho_N)\|u\|_{L^2(\Omega)}^2 
 \end{equation}
 
for every $u \in \h(\Omega)$.
\end{remark}
\begin{lemma}
    \label{ps-lem}
     Assume \eqref{assump w0}-\eqref{assump w2} hold true. Then, $I|_{M}$ satisfies the Palais Smale condition at the level $\lambda_k.$
    \end{lemma}
\begin{proof}
Let $k \in \N$, $(u_n)_{n \in\N}\subseteq M$ and $(\delta_n)_{n \in \N}\subseteq\R$ be such that
\begin{equation}
\label{1-1}
I(u_n)=\mathcal{E}_L(u_n,u_n)\to \lambda_k  \geq  \lambda_0
\end{equation} 
and
\[I'(u_n)-\delta_n J'(u_n)\to 0 \quad \text{in }(\h(\Omega))^\ast \]
{\it i.e.} for any $v \in \h(\Omega)$,
\begin{equation}
\label{PS}
\mathcal{E}_L(u_n,v)-\delta_n\into \omega(x) u_nv~dx \to 0 \quad \text{as} \ n\to \infty
\end{equation}
where $(\h(\Omega))^\ast$ denotes the dual of $\h(\Omega).$
Existence of such a sequence follows from Lemma \ref{energy bdd below} and \cite[Lemma 3.3]{Ekeland}.
From \eqref{1-1}, we get, for $\epsilon>0$, there exists $n_0(\epsilon)\in \N$ such that
\[\mathcal{E}_L(u_n,u_n)\leq \lambda_k+\epsilon \quad \text{whenever } n\geq n_0(\epsilon). \]
Now, by using Theorem \ref{Hardy-inequ}, \eqref{assump w2} and the fact that $u_n \in M$, we obtain 
    \[
    \begin{split}
    \lambda_k+\epsilon \geq \mathcal{E}_L(u_n,u_n)&\geq\int_{\R^N} u_n^2 \ln\l(\frac{1}{|x|^2}\r) ~dx+\l(\ln4+2\psi(\frac{N}{4})\r)\int_{\R^N} u_n^2 ~dx\\&\geq \alpha\into u_n^2 ~dx+\lambda_0,
 \quad \text{whenever } n \geq n_0(\epsilon).    \end{split}
    \]
This implies
    \[\|u_n\|^2_{L^2(\Omega)}\leq \l(\frac{\lambda_k+\epsilon - \lambda_0}{\alpha}\r).\]
Next, in view of \eqref{E-L-bd-below-above}, above estimate further gives
\begin{equation}\label{est-bounded}
    \mathcal{E}(u_n,u_n) \leq  \lambda_k+\epsilon+  \max\{(c_N|\Omega|-\rho_N),0\}\l(\frac{\lambda_k+\epsilon-\lambda_0}{\alpha}\r).
\end{equation}
Hence, $(u_n)_{n \in \N}$ is bounded in $\h(\Omega)$. 
Therefore, up to a subsequence (keeping the same notation), by Theorem \ref{cpct-embed}, we get 
\[
u_n\rightharpoonup u \text{ in } \h(\Omega)
\quad \text{and}\quad
u_n\to u \text{ in } L^\eta(\Omega)
\]
for some $u\in \h(\Omega).$ Using Lemma \ref{limit passage} (i), we get
    \begin{equation}\label{1-2}1= \lim_{n \to \infty} \into \omega(x)  u_n^2 ~dx = \into \omega(x)  u^2 ~dx.\end{equation}
Therefore, $u\in M.$ 
Now, by taking $v=u_n$ in \eqref{PS} and using \eqref{1-1}, \eqref{est-bounded} and \eqref{1-2}, we obtain
\[
\delta_n = \mathcal{E}_L(u_n,u_n) + o(1) \to \lambda_k \ \text{as} \ n \to \infty.
\]
Moreover, inserting $v=u_n-u$ in \eqref{PS} and using \eqref{est-bounded}, we get \begin{equation}\label{1-3}
\mathcal{E}_L(u_n,u_n-u)-\delta_n\into \omega(x) u_n(u_n-u)~dx\to 0 \quad n \to \infty.
\end{equation}
By Lemma \ref{limit passage} (iii) and \eqref{1-2}, we get
\[\into \omega(x) u_n(u_n-u)~dx\to 0 \quad \text{as } n \to \infty.\]
This, \eqref{1-3} and the fact that $\delta_n \to \lambda_k$, yields, $\mathcal{E}_L(u_n,u_n-u)\to 0$ as $n\to\infty.$
Finally, by the using $(S)$ property of $I$ (see \cite[Lemma 3.3]{Arora-Hajaiej-Perera}), upto a subsequence, we get $\|u_n-u\|_{\h(\Omega)}\to 0$ as $n \to \infty$. Hence, the required claim. 
\end{proof}

Building on the above results, we prove the existence of an unbounded sequence of eigenvalues corresponding to \eqref{equ}, thereby showing that the spectrum diverges to infinity.
 
\begin{theorem}
    \label{exist-eigenval-L}
 Assume \eqref{assump w0}-\eqref{assump w2}  hold true. The problem \eqref{equ} admits a sequence of eigenvalues 
    \[\lambda_0 \leq  \lambda_1\leq\lambda_2\leq \cdots \leq \lambda_k\leq \cdots\] that diverges to $+\infty,$ where $\lambda_k$ is defined in \eqref{lambda}
    and $\lambda_0$ is the constant in \eqref{assump w2}.
\end{theorem}

\begin{proof}
We first prove that $\lambda_k$ defined in \eqref{lambda} is a critical value {of $I$}. Fix $k \in \N$. On the contrary, suppose that $\lambda_k$ is a regular value of $I$. 
By the deformation theorem on $C^1$-manifolds (see Lemma \ref{def-lem} {(iii)}), there exist $ \eps > 0$ and an odd homeomorphism $\Theta : M \to M$ such that
\begin{equation}
\label{odd-homeo}
I(\Theta(u)) \leq \lambda_k-\eps \quad \text{for all} \ u \in M \quad \text{if} \quad I(u) \leq \lambda_k+\eps.
\end{equation}
By \eqref{lambda}, there exists $A\in F_k$ such that \[\sup_{u \in A} I(u)\leq \lambda_k+\eps.\]
Using Lemma \ref{genus-prop} (ii) with $A=A$ and $B=\Theta(A)$, we have $\gamma(\Theta(A))\geq\gamma(A)\geq k$. This further gives in view of \eqref{odd-homeo}
\[ \Theta(A) \in F_k \quad \text{and} \quad \sup_{u \in \Theta(A)}I(u) \leq \lambda_k-\eps,\] which contradicts \eqref{lambda}.  Therefore, $\lambda_k$ is a critical value of $I$. Now, by Ekeland's variational principle, there exist $(\delta_n)_{n \in \N}\subseteq\R$ and Palais-Smale sequence $(u_n)_{n \in\N} \subseteq M$ at critical value $\lambda_k$. 
By applying Lemma \ref{ps-lem}, the above Palais-Smale sequence $u_n$ has strongly convergent subsequence, {\it i.e.,} there exists a $u \in M$ such that $u_n \to u$ (upto a subsequence) in $\h(\Omega)$ as $n\to\infty.$ Next, we show that
    \begin{equation}\label{est:ray-2}
        \mathcal{E}(u_n, u_n) \to \mathcal{E}(u, u) \quad \text{and} \quad \lim_{n \to \infty} \intb \frac{u_n(x)u_n(y)-u(x)u(y)}{|x-y|^N} ~dx ~dy.
    \end{equation}
The first claim follows directly from the definition of the $\h(\Omega)$-norm. To establish the second claim, we proceed as follows. Using the symmetry property, Lemma \ref{est:lower-order-term} and the fact that $u_n \to u$ in $L^2(\Omega)$, we obtain
\begin{equation}\label{est:lower:conv}
    \begin{split}
        \intb & \frac{u_n(x)u_n(y)}{|x-y|^N} ~dx ~dy - \intb \frac{u(x)u(y)}{|x-y|^N} ~dx ~dy\\
        & = \intb \frac{u_n(y)(u_n-u)(x)+u(x)(u_n-u)(y)}{|x-y|^N} ~dx ~dy\\
        & \leq |\Omega|\l(\|u_n\|_{L^2(\Omega)}\|u_n-u\|_{L^2(\Omega)}+\|u\|_{L^2(\Omega)}\|u_n-u\|_{L^2(\Omega)} \r) \to 0,
    \end{split}
\end{equation}
  as $n \to \infty$. Hence the claim. Now, by using \eqref{est:ray-2}, Theorem \ref{cpct-embed} and Lemma \ref{limit passage}, we conclude that $u\in M$ satisfies \[\mathcal{E}_L(u,u)=\lambda_k \into \omega(x) u^2 ~dx.\]
Now, $\lambda_k\leq \lambda_{k+1}$ follows from the fact that $F_{k+1}\subset F_{k}$ and the definition of $\lambda_k$. Next, we prove that $\lambda_k \to \infty$ as $k \to \infty.$ Suppose, for contradiction, there exists a $\beta>0$ such that $ \lambda_k\leq \beta$, for all $k \in \N.$ Since $\h(\Omega)$ is separable, by \cite[Theorem 1]{Ovsepien-Pelczynski}, there exists a sequence $(w_k, w_k^\ast)_{k \in \N} \subset \h(\Omega) \times (\h(\Omega))^\ast$ with the following properties
\begin{equation}\label{seq-prop}
\h(\Omega) = \overline{\sp\{w_k: k \in \mathbb{N}\}} \quad \text{and} \quad \langle w_k^\ast, w_j \rangle_{(\h(\Omega))^\ast, \h(\Omega)} = \delta_{i,j} \ \text{for all} \ i,j \in \N.
\end{equation}
Moreover, if $\langle w_k^\ast, v \rangle =0$ for all $n \in \mathbb{N}$ implies that $v=0.$
Let
\[
H_k = \overline{\sp\{w_k, w_{k+1}, \cdots\}}
\quad \text{and} \quad
a_k = \inf_{A \in F_k} \; \sup_{u \in A \cap H_k} \mathcal{E}_L(u,u).
\]
Note that the co-dimension of $H_k$ is $(k-1)$. By Lemma \ref{genus-prop} (iii), we have
\[
A \cap H_k \neq \emptyset \quad \text{for all } A \in F_k.
\]
Furthermore,
\[
a_k \leq \lambda_k \leq \beta, \quad \text{ for all } k \in \mathbb{N}.
\]
Now, for each $k \in \mathbb{N}$, choose $v_k \in A \cap H_k$ such that
\[
\int_\Omega \omega(x) v_k^2 ~dx = 1
\quad \text{and} \quad
 a_k \leq \mathcal{E}_L(v_k,v_k) \leq \beta+1.
\]
Now, by using Theorem \ref{Hardy-inequ}, \eqref{assump w2} and the fact that $v_k \in M$, we obtain 
    \[
    \begin{split}
    \beta + 1 \geq \mathcal{E}_L(v_k, v_k)&\geq\int_{\R^N} v_k^2 \ln\l(\frac{1}{|x|^2}\r) ~dx+\l(\ln4+2\psi(\frac{N}{4})\r)\int_{\R^N} v_k^2 ~dx\\
    &\geq \alpha\into v_k^2 ~dx+\lambda_0,
 \quad \text{for all} \ k \in \mathbb{N}.    
 \end{split}
    \]
This implies
    \[\|v_k\|^2_{L^2(\Omega)}\leq \l(\frac{\beta +1 - \lambda_0}{\alpha}\r) \]
and
\[
  \mathcal{E}(v_k,v_k) \leq  \beta + 1 + \max\{(c_N|\Omega|-\rho_N),0\} \l(\frac{\beta + 1 -\lambda_0}{\alpha}\r).
\]
Hence, $(v_k)_{k \in \N}$ is bounded in $\h(\Omega)$. 
Therefore, by using the reflexivity of $\h(\Omega)$ and Lemma \ref{limit passage} (i), we can ensure the existence of an element 
$0\not\equiv v \in \h(\Omega)$ such that
\[
v_k \rightharpoonup v \quad \text{in } \h(\Omega)
\text{ with }
\into \omega(x) v^2 ~dx = 1.
\]
However, by the choice of the sequence $(w_k, w_k^\ast)_{k \in \mathbb{N}}$ satisfying \eqref{seq-prop}, we have
\[
w_m^\ast(v) = \lim_{k \to \infty} w_m^\ast(v_k) = 0,
\quad \text{for every } m \in \mathbb{N},
\]
which implies that $v = 0$. This is a contradiction. Hence,
\[
\lambda_k \to \infty \quad \text{as } k \to \infty.
\]
This completes the proof.

\end{proof}

In what follows, $e_1$ (called first eigenfunction) denotes the eigenfunction corresponding to the first eigenvalue $\lambda_1$ in $\Omega$ (defined in \eqref{lambda}). 

Next, we establish the existence of the principal eigenvalue associated with \eqref{equ} by means of the Rayleigh quotient and show that it coincides with the first Lusternik-Schnirelmann eigenvalue.

\begin{theorem}
\label{existence lambda1}
 Assume \eqref{assump w0}-\eqref{assump w2} hold true. The infimum\begin{equation}
    \label{P1}
    \tag{P$_1$}
     \lambda_1^R=\inf_{u\in M}\ \mathcal{E}_L(u,u),
\end{equation} is achieved by some $e_1^R \in M$. 
\end{theorem}

\begin{proof}
    Let {$(u_n)_{n\in \N} \subseteq M$} be a minimizing sequence for \eqref{P1}, existence of such a sequence follows from Lemma \ref{energy bdd below}. Thus, there exists a constant $C_1>0$ such that
    $\mathcal{E}_L(u_n,u_n)\leq C_1$. Now, by repeating the same arguments as in Theorem \ref{exist-eigenval-L}, we get $(u_n)_{n \in \N}$ is bounded in $\h(\Omega)$. By reflexivity of $\h(\Omega)$ and Theorem \ref{cpct-embed}, for some $u\in \h(\Omega)$ we get $u_n \rightharpoonup u$ in $\h(\Omega),$ $u_n \to u$ in $L^\eta(\Omega)$ and a.e. in $\Omega$, upto a subsequence. By the weak lower semi-continuity of the $\h(\Omega)$ norm, we have
    \begin{equation}\label{est:ray-1}
       \mathcal{E}(u,u) \leq \liminf_{ n\to \infty} \mathcal{E}(u_n,u_n). 
    \end{equation}
Now, by combining \eqref{est:ray-1} and \eqref{est:ray-2}, we have 
    \begin{equation}\label{4-1}\mathcal{E}_L(u,u) \leq \liminf_{ n \to \infty} \mathcal{E}_L(u_n,u_n)=\inf_{u\in M} \mathcal{E}_L(u,u)=\lambda_1^R.\end{equation}
   Using Lemma \ref{limit passage} (i), we get
    \begin{equation} \label{4-2}1= \lim_{n \to \infty} \into \omega(x)  u_n^2 ~dx = \into \omega(x)  u^2 ~dx.\end{equation}
    From \eqref{4-1} and \eqref{4-2}, we obtain $\lambda_1^R=\mathcal{E}_L(u,u)$.
Let $v\in C_c^\infty(\Omega)$, then using Lagrange multipliers, we have
    \[
    \begin{split}
        0&= \frac{d}{d\eps}\bigg|_{\eps=0} \l[\mathcal{E}_L(u+\eps v, u+\eps v)-\lambda \l(\into \omega(x) (u+\eps v)^2 ~dx-1\r)\r]\\
        &=\lim_{\eps \to 0} \l[\frac{\mathcal{E}_L(u+\eps v, u+\eps v)-\mathcal{E}_L(u,u)}{\eps}-\frac{\lambda}{\eps} \into \omega(x) \l[(u+\eps v)^2-u^2\r] ~dx\r]\\
        & = 2 \mathcal{E}_L(u,v)-2 \lambda \into \omega(x)  uv~dx,
    \end{split}
    \]
    {for some $\lambda$. Hence $u$ is an eigenfunction corresponding to eigenvalue $\lambda$ and we denote it by $e_1^R$, which finishes the proof.}
\end{proof}

\begin{remark}
\label{Rayleigh-char}
The Rayleigh characterization of the first eigenvalue is
\[
\lambda_1^R=\min_{u\in M}\mathcal{E}_L(u,u).
\]
Note that $e_1^R$ exists by Theorem \ref{existence lambda1}.
The first Lusternik Schnirelmann eigenvalue $\lambda_1$ is defined by \eqref{lambda} for $k=1.$
By Lemma \ref{genus-prop} (i), the set $A_0 = \{\pm e_1^R\}$ has $\gamma(A_0)=1$. Hence, $A_0$ belongs to $F_1$. Then,
\[
\inf_{A \in F_1}\sup_{v \in A} \mathcal{E}_L(v,v)\leq\sup_{v \in A_0} \mathcal{E}_L(v,v) = \mathcal{E}_L(e_1^R,e_1^R),
\]
yields
\[
\lambda_1 \leq \lambda_1^R.
\]
Conversely, for any $A \in F_1$, by definition
\[
\sup_{u \in A} \mathcal{E}_L(u,u) \geq \inf_{u \in M} \mathcal{E}_L(u,u) = \lambda_1^R,
\]
and taking the infimum over $A \in F_1$ gives
\[
\lambda_1 \geq \lambda_1^R.
\]
 Therefore, the variational characterization of $\lambda_1$ via \eqref{lambda} and the one due to Rayleigh coincide. {Consequently, $e_1$ and $e_1^R$ coincide.}
\end{remark}

\section{Sign properties of eigenfunctions and nodal domain type estimate}
\label{sign properties}
In this section, we show that the first eigenfunction $e_1$ has constant sign, whereas eigenfunctions corresponding to higher eigenvalues change sign. Under additional assumptions, we also establish nodal domain type estimates for higher eigenfunctions.

\begin{proposition}
    \label{constant sign e-1}
 First eigenfunction $e_1$ does not change sign in $\Omega.$
\end{proposition}
\begin{proof}
    In view of Remark \ref{Rayleigh-char} and \cite[Lemma 3.3]{Chen-Weth}, we get 
    \[\lambda_1 \leq \mathcal{E}_L(|e_1|,|e_1|) \leq \mathcal{E}_L(e_1,e_1) =\lambda_1.\]
    Thus, $\mathcal{E}_L(|e_1|, |e_1|)=\mathcal{E}_L(e_1,e_1).$
    Hence, using \cite[Lemma 3.3]{Chen-Weth}, we get that $e_1$ does not change sign in $\Omega.$ 
\end{proof}

\begin{theorem}
\label{sign-changing}
Let $\omega \in L^\infty(\Omega).$ Then, eigenfunction $u$ associated to eigenvalue $\lambda \neq \lambda_1$ to the problem \eqref{equ} changes sign in $\Omega.$
\end{theorem}
\begin{proof}
Without loss of generality, we can assume that $e_1\geq 0$ a.e. in $\Omega$ and $\lambda>\lambda_1$. Suppose, by contradiction, $u \geq 0$ in $\Omega.$  By \cite[Theorem 1.1]{Jarohs-Weth} and \cite[Theorem 1.9]{Dyda-Jarohs-Sk}, $e_1,u>0$ in $\Omega$ and
$e_1,u \in L^\infty(\Omega).$ Moreover, by using \cite[Theorem 1.11]{Chen-Weth} and \cite[Corollary 5.3]{Santamaria-Rios-Saldana}, we get $\frac{e_1}{u}, \frac{u}{e_1} \in L^\infty(\Omega).$
Now, for all $t \in [0,1]$, define $v_t=te_1^2+(1-t)u^2.$ From \cite[Proposition 5.2]{Arora-Giacomoni-Vaishnavi}, we obtain
        \begin{equation}
        \label{1}
    \mathcal{E}_L(v_t^{\frac{1}{2}},v_t^{\frac{1}{2}})-\mathcal{E}_L(u,u)\leq t(\mathcal{E}_L(e_1,e_1)-\mathcal{E}_L(u,u)).
        \end{equation}
Furthermore, by the convexity of the map $t\mapsto f(t):=\mathcal{E}_L(v_t^{\frac{1}{2}},v_t^{\frac{1}{2}}),$ we have
        \[f(t)-f(0)\geq f'(0)t \quad \text{for every } t \in [0,1].\]
By using \cite[(5.3)]{Arora-Giacomoni-Vaishnavi} and the above convexity inequality, we get
\begin{equation}
\label{2}
\mathcal{E}_L(v_t^{\frac{1}{2}},v_t^{\frac{1}{2}})-\mathcal{E}_L(u,u)\geq t\mathcal{E}_L\l(u,\frac{u^2-e_1^2}{u}\r).
\end{equation}
Combining \eqref{1} and \eqref{2}, we obtain
\[
\mathcal{E}_L\l(u,\frac{u^2-e_1^2}{u}\r)\leq \mathcal{E}_L(e_1,e_1)-\mathcal{E}_L(u,u). 
\]
Now, by using the fact that $e_1, u \in M$ and the weak formulation \eqref{weak-formulatn}, we have
\[0=\lambda\into \omega(x)(u^2-e_1^2)~dx\leq \lambda_1\into \omega(x) e_1^2 ~dx-\lambda\into \omega(x) u^2 ~dx=(\lambda_1-\lambda),\] 
{\it i.e.,} $\lambda\leq \lambda_1$, a contradiction to $\lambda>\lambda_1.$ 
Applying the same arguments on $u\leq 0$ again leads to the same contradiction. Hence, $u$ must change sign in $\Omega.$
\end{proof}

Denote 
\begin{equation}\label{def:posit-negat}
    \Omega^\pm_u=\{x\in \Omega:u(x)\gtrless 0\}.
\end{equation}
Next, we state and prove a nodal domain type estimate on $\Omega^\pm_u$ which also plays a key role in deriving the properties of first eigenvalue in the next section.  
\begin{theorem}
\label{nodal estimate}
 Let $\psi$ satisfies \eqref{psi}-\eqref{eta} and $\vartheta_1, \vartheta_2$ be Orlicz functions in $\Phi_c$ such that \[(\vartheta_{1})^{-1}_{\text{left}}(x) (\vartheta_2)^{-1}_{\text{left}}(x)\leq (\psi^\ast)^{-1}_{\text{left}}(x)\]
and $\omega\in L^{\psi^\ast}(\Omega)\cap L^{\vartheta_1}(\Omega)$. Let $v$ be an eigenfunction associated to an eigenvalue $\lambda\neq \lambda_1$. Then,
   \[|\lambda|\geq \frac{\vartheta_2^{-1}(\frac{1}{|\Omega^\pm_v|})\l(1-\max\{(c_N|\Omega^\pm_v|-\rho_N),0\}S_2\r)}{2S_1\|\omega\|_{L^{\vartheta_1}(\Omega)}},\]
      where $S_1, S_2$ are the best constants of the embeddings $\h(\Omega)\hookrightarrow L^\ph(\Omega)$ and $\h(\Omega) \hookrightarrow L^2(\Omega),$ respectively.
\end{theorem}
\begin{proof}
Let $v$ be an eigenfunction associated to an eigenvalue $\lambda\neq \lambda_1$. Set $A(t)=\vartheta_1(t)$, $B(t)=\vartheta_2(t)$ and $C(t)=\psi^\ast(t)$ in \cite[Theorem 2.3]{Neil} and using the definition of Orlicz norm in \eqref{orlicz-norm-def}, we get 
      \begin{equation} \label{orlicz-est}\|\omega\|_{L^{\psi^\ast}(\Omega^+_v)}\leq2\|\omega\|_{L^{\vartheta_1}(\Omega^+_v)}\|1\|_{L^{\vartheta_2}(\Omega^+_v)}\leq 2 \|\omega\|_{L^{\vartheta_1}(\Omega^+_v)}\l(\frac{1}{\vartheta_2^{-1}(\frac{1}{|\Omega^+_v|})}\r).\end{equation}

By taking $v^+$ as a test function in \eqref{weak-formulatn}, we get
\[
\begin{split}
\mathcal{E}_L(v,v^+)=\lambda\int_{\Omega} \omega(x) (v^+)^2 ~dx.
\end{split}
\]
Using $v^+(x)v^{-}(x)=0,$ we get
\begin{equation}\label{sign-ineq}
    (v(y)-v(x))(v^+(y)-v^+(x))=(v^+(y)-v^-(x))^2+v^+(x)v^-(y)+v^+(y)v^-(x).
\end{equation}
Now, using the definition of $\mathcal{E}_L(\cdot,\cdot)$ and $v= v^+ - v^-$, we obtain
\begin{align}
\label{est-2}
\begin{split}
&\mathcal{E}_L(v,v^+)=\bigg(\frac{c_N}{2}{\iint\limits_{\substack{x,y\in\R^N \\ \abs{x-y} \leq 1}}} \frac{|v^+(x)-v^+(y)|^2}{|x-y|^N} ~dx ~dy+c_N{\iint\limits_{\substack{x,y\in\R^N\\ \abs{x-y} \leq 1}}} \frac{v^+(x)v^-(y)}{|x-y|^N} ~dx ~dy\\&\qquad+\rho_N \int_{\Omega^+} v^2 ~dx+c_N{\iint\limits_{\substack{x,y\in\R^N \\ \abs{x-y} \geq 1}}} \frac{v^-(x)v^+(y)}{|x-y|^N} ~dx ~dy-c_N{\iint\limits_{\substack{x,y\in\R^N \\ \abs{x-y} \geq 1}}} \frac{v^+(x)v^+(y)}{|x-y|^N} ~dx ~dy\bigg)\\&\qquad=\lambda \int_{\Omega} \omega(x) (v^+)^2 ~dx.
\end{split}
\end{align}
It is easy to see that
\begin{equation}\label{est-3}
    \begin{split}
\mathcal{E}_L(v,v^+)&\geq \frac{c_N}{2}{\iint\limits_{\substack{x,y\in\R^N \\ \abs{x-y} \leq 1}}} \frac{|v^+(x)-v^+(y)|^2}{|x-y|^N} ~dx ~dy + \rho_N \|v^+\|_2^2  \\
& \qquad \qquad -c_N{\iint\limits_{\substack{x,y\in\R^N \\ \abs{x-y} \geq 1}}} \frac{v^+(x)v^+(y)}{|x-y|^N} ~dx ~dy =\mathcal{E}_L(v^+,v^+).
\end{split}
\end{equation}
Combining \eqref{est-2} and \eqref{est-3}, we get
\begin{equation}\label{est-4}
    \mathcal{E}_L(v^+, v^+) \leq \lambda \int_{\Omega} \omega(x) (v^+)^2 ~dx.
\end{equation}
Now, by using \eqref{E-L-bd-below-above}, \eqref{est-4}, H\"older inequality, \eqref{orlicz-est}, Lemma \ref{lem:norm-est} and Theorem \ref{cpct-embed}, we further get
\[
\begin{split}
 \mathcal{E}(v^+,v^+) & -(c_N|\Omega^+|-\rho_N) \int_{\Omega^+} (v^+)^2 ~dx \leq \mathcal{E}_L(v^+,v^+) \\&\leq \lambda \int_{\Omega} \omega(x) (v^+)^2 ~dx \leq 2 |\lambda|\|\omega\|_{L^{\psi^\ast}(\Omega^+_v)} \|v^2\|_{L^{\psi}(\Omega^+_v)}\\
 & \leq 2|\lambda|\|\omega\|_{L^{\vartheta_1}(\Omega^+_v)}\l(\frac{1}{\vartheta_2^{-1}(\frac{1}{|\Omega^+_v|})}\r) \|v^+\|_{L^\ph(\Omega^+_v)}^2 \\
 & \leq 2S_1|\lambda|\|\omega\|_{L^{\vartheta_1}(\Omega)}\l(\frac{1}{\vartheta_2^{-1}(\frac{1}{|\Omega^+_v|})}\r) \mathcal{E}(v^+,v^+)
\end{split}
\]
where $S_1$ is the best constant of the embedding of $\h(\Omega)$ in $L^\ph(\Omega)$.

Using $\h(\Omega)\hookrightarrow L^2(\Omega),$ we get
 \[1\leq 2S_1|\lambda|\|\omega\|_{L^{\vartheta_1}(\Omega)}\l(\frac{1}{\vartheta_2^{-1}(\frac{1}{|\Omega^+_v|})}\r)+  \max\{(c_N|\Omega^+_v|-\rho_N),0\} S_2 \]
 where $S_2$ is the best constant of the embedding of $\h(\Omega) $ in $L^2(\Omega).$ Hence \[|\lambda|\geq \frac{\vartheta_2^{-1}(\frac{1}{|\Omega^+_v|})\l(1- \max\{(c_N|\Omega^+_v|-\rho_N),0\}S_2\r)}{2S_1\|\omega\|_{L^{\vartheta_1}(\Omega)}}.\]
By repeating the same arguments, the above holds for $\Omega^-$ as well. 
\end{proof}
\section{Qualitative properties and characterization of eigenvalues}
\label{behav-char}
\subsection{On the first eigenvalue}
In this subsection, we prove some properties of the first eigenvalue.

\begin{proposition}
\label{simplicity}
 The first eigenvalue $\lambda_1$ is simple. 
\end{proposition}

\begin{proof}
    Let $u, v$ be two eigenfunctions corresponding to $\lambda_1.$ Since $u,  v$ do not change sign in $\Omega,$ we can assume $u\geq 0$ and $v\geq 0$ {a.e.} in $\Omega.$ Set $z= \l(\frac{u^2+v^2}{2}\r)^{\frac{1}{2}} \in \mathbb H(\Omega)$. It is easy to see that $\into \omega(x)  z^2=1,$ {{\it i.e.,} $z\in M$.}
    Using \cite[Lemma 13]{Lindgren-Lindqvist}, we have
    \[|z(y)-z(x)|^2 \leq \frac{1}{2}|u(y)-u(x)|^2 + \frac{1}{2}|v(y)-v(x)|^2\] with equality only {when $(x,y)\in \R^N\times \R^N$} satisfies $u(x)v(y)=u(y)v(x).$
   
    By Cauchy-Schwarz's inequality, we have
    \[u(x)u(y)+v(x)v(y) \leq \l(u^2+v^2\r)^{\frac{1}{2}}(x)\l(u^2+v^2\r)^{\frac{1}{2}}(y),\]
    {\it i.e.},
    \[\frac{u(x)u(y)}{2}+\frac{v(x)v(y)}{2} \leq z(x)z(y),\]
    with equality only {when $(x,y)\in \R^N\times\R^N$} satisfies $u(x)=\alpha u(y)$ and $v(x)=\alpha v(y),$ $\alpha\in \R.$
    Thus,
    \[
    \begin{split}
        \mathcal{E}_L(z,z) &= \frac{c_N}{2}\inta \frac{|z(x)-z(y)|^2}{|x-y|^N} ~dx ~dy -c_N \intb \frac{z(x)z(y)}{|x-y|^N} ~dx ~dy\\
        &\qquad \qquad \qquad\qquad\qquad+\rho_N \int_{\R^N} |z|^2 ~dx\\
        &\leq \frac{c_N}{2}\l(\frac{1}{2}\inta \frac{|u(x)-u(y)|^2}{|x-y|^N} ~dx ~dy+ \frac{1}{2}\inta \frac{|v(x)-u(y)|^2}{|x-y|^N} ~dx ~dy\r)\\
        & \quad -\frac{c_N}{2}\l(\frac{1}{2}\intb \frac{u(x)u(y)}{|x-y|^N} ~dx ~dy+\frac{1}{2}\intb \frac{v(x)v(y)}{|x-y|^N} ~dx ~dy \r)\\
        & \qquad \qquad \qquad +\rho_N \int_{\R^N} \l(\frac{u^2}{2}+\frac{v^2}{2}\r)~dx\\
        &= \frac{\mathcal{E}_L(u,u)}{2}+ \frac{\mathcal{E}_L(v,v)}{2}= \frac{\lambda_1}{2}+\frac{\lambda_1}{2} = \lambda_1.
    \end{split}
    \]
    Moreover, by the definition of $\lambda_1$, we get $\lambda_1 \leq \mathcal{E}_L(z,z) \leq \lambda_1,$ which implies, $\mathcal{E}_L(z,z)=\lambda_1.$ Hence, \[\frac{u(x)}{u(y)}=\frac{v(x)}{v(y)}=\alpha,\] which further implies, $u(x)=\alpha v(x).$ This proves that $\lambda_1$ is simple. 
    \end{proof}

\begin{proposition}
 Assume \eqref{assump w0}-\eqref{assump w2} hold true and $\omega \in L^\infty(\Omega)$. The first eigenvalue  $\lambda_1$ is isolated.
\end{proposition}

\begin{proof}
By Proposition \ref{constant sign e-1} and \cite[Theorem 1.1]{Jarohs-Weth},  we get $e_1>0$ in $\Omega.$ Suppose $\lambda_1$ is not isolated. Then, there exists a sequence of eigenvalues $(\beta_n)_{n\in\N}$ such that $\beta_n \to \lambda_1$ with $\beta_n>\lambda_1$, for every $n\in\N.$ {Let $(u_n)_{n\in\N}\subseteq M$ be a sequence of eigenfunctions corresponding to $(\beta_n)_{n \in\N}$ {\it i.e.} for any $\phi \in \h(\Omega)$
\[
\mathcal{E}_L(u_n, \phi) = \beta_n \int_{\Omega} \omega(x) u_n \phi ~dx \quad \text{and} \quad \mathcal{E}_L(u_n,u_n)=\beta_n\into \omega(x) u_n^2 ~dx = \beta_n \overset{n \to \infty}{\to} \lambda_1. \] 
Using Theorem \ref{Hardy-inequ} and \eqref{assump w2}, for $\epsilon>0,$ there exists a $\tilde{N}(\epsilon) \in \N$ such that 
    \[
    \begin{split}
    \lambda_1+\epsilon \geq \mathcal{E}_L(u_n,u_n)&\geq\int_{\R^N} u_n^2 \ln\l(\frac{1}{|x|^2}\r) ~dx+\l(\ln4+2\psi(\frac{N}{4})\r)\int_{\R^N} u_n^2 ~dx\\&\geq \alpha\into u_n^2 ~dx+\lambda_0,
    \end{split}
    \]
when $n\geq\tilde{N}(\epsilon).$ Hence, $\|u\|_{L^2(\Omega)}^2 \leq \hat{C}$ where $\hat{C}>0$ is a constant depending on $\epsilon, \lambda_1, \alpha,\lambda_0.$ Moreover, in view of \eqref{E-L-bd-below-above}}, we obtain $(u_n)_{n \in \N}$ is bounded in $\h(\Omega)$. Therefore, up to a subsequence, by Theorem \ref{cpct-embed}, we get 
$ u_n\rightharpoonup u \text{ in } \h(\Omega)
\ \text{and}\
u_n\to u$ in $L^\eta(\Omega)
$ and a.e. in $\Omega.$ Note that, for any $\phi \in \h(\Omega)$, $w \mapsto \mathcal{E}_L(w, \phi)$ is a bounded linear functional on $\h(\Omega)$. Since $u_n\rightharpoonup u$ in $\h(\Omega)$, we obtain
\[
\mathcal{E}_L(u_n, \phi) \to \mathcal{E}_L(u, \phi) \quad \text{for every} \ \phi \in \h(\Omega).
\]

Finally, by Lemma \ref{est:lower-order-term} and Lemma \ref{limit passage} (i), we can conclude that the limit $u$ lies in $M$ and satisfies \[\mathcal{E}_L(u,u) = \lambda_1 \into \omega(x) u^2~dx,\] {{\it i.e.,} $u$ is an eigenfunction associated to $\lambda_1$. Now, by Proposition \ref{simplicity}, we know that $\lambda_1$ is simple, thus, $u \in$ span$\{e_1\}.$ As $u\in M$, it follows that $u=\pm e_1$. Since $e_1>0$ in $\Omega$, $u$ is either strictly positive or strictly negative in $\Omega.$  By Theorem \ref{nodal estimate}, we have
\begin{equation}\label{eq}1\leq 2S_1|\lambda|\|\omega\|_{L^{\vartheta_1}(\Omega)}\l(\frac{1}{\vartheta_2^{-1}(\frac{1}{|\Omega^\pm_n|})}\r) + \max\{(c_N|\Omega^+_n|-\rho_N),0\} S_2, \end{equation} where $\Omega^\pm_n$ is the nodal domain (defined in \eqref{def:posit-negat}) corresponding to each $u_n.$ Since $u \in \{e_1, -e_1\}$, then either $\lim\limits_{n \to \infty} \Omega^+_n =0$ or $\lim\limits_{n \to \infty} \Omega^-_n =0$. Using the fact that $\vartheta_2$ is increasing and $\lim\limits_{t \to\infty} \vartheta_2(t)=\infty$, passing $n \to \infty$ in \eqref{eq}, we get \[1\leq \max\{-\rho_N ,0\}S_2<0,\] 
this contradiction shows that $\lambda_1$ is isolated.
}  
\end{proof}

\subsection{On the second eigenvalue}
In this part of this work, we give an alternative variational characterization to the second eigenvalue $\lambda_2$ of \eqref{equ} analogous to that introduced in \cite{Iannizzotto} for the fractional Laplacian case.

\begin{proposition}
\label{isolate-alter}
  It holds that 
  \[\lambda_2=\min\{\lambda>\lambda_1:\lambda \text{ is an eigenvalue of } \eqref{equ}\}.\]
\end{proposition}

\begin{proof}
    From Theorem \ref{exist-eigenval-L}, we have $\lambda_1 \leq \lambda_2$. On contrary, assume that $\lambda_1 = \lambda_2$. By \eqref{lambda}, for any $n \in \mathbb{N},$ we can find $A_n \in F_2$ such that
\[
\sup_{u \in A_n} \mathcal{E}_L(u,u) \leq \lambda_1 + \frac{1}{n}.
\]
For all $\rho > 0$, we define two relatively open subsets of $M$ as
\[
\mathcal{S}^\pm_\rho =
\l\{ u \in M :
\l| \into \omega(x)(u \pm e_1)^2 ~dx \r| < \rho \r\}
\]
where $e_1$ is first eigenfunction corresponding to the eigenvalue $\lambda_1$. Next, we two cases: \newline
\textbf{Case 1:} If there exists a sequence $(\rho_n)_{n\in \N}$ such that $\rho_n \to 0^+$ as $n \to \infty$ and for all $n \in \N$,
\[
\{ u \in M: \mathcal{E}_L(u,u)\leq \lambda_1 + 1/n \}
\subseteq \mathcal{S}^+_{\rho_n} \cup \mathcal{S}^-_{\rho_n}.
\]
Since $A_n$ is symmetric, we have $A_n \cap \mathcal{S}^\pm_{\rho_n} \neq \emptyset$ and there exists a $n_0 \in \mathbb{N}$ such that
\[
A_n \cap \mathcal{S}^+_{\rho_n} \cap \mathcal{S}^-_{\rho_n} = \emptyset \quad \ n \geq n_0.
\]
If the above is not true, there would exist a sequence $(u_n)_{n\in \N}\subseteq M$ such that for all $n \in \N$
\begin{equation}\label{est:lambda2-chara}
    \mathcal{E}_L(u_n,u_n) \leq \lambda_1 + \frac{1}{n}
\quad 
\text{and}\quad
\l| \into \omega(x)(u_n \pm e_1)^2 ~dx \r| < \rho_n.
\end{equation}
Using \eqref{E-L-bd-below-above}, \eqref{assump w2} and \eqref{est:lambda2-chara}, we see that $(u_n)_{n\in \N}$ is bounded in $\h(\Omega)$, hence upto a subsequence, we have $u_n \rightharpoonup u$ in $\h(\Omega)$. Now, by using the weak lower semicontinuity of $\h(\Omega)$ norm, Lemma \ref{limit passage} (i) and Theorem \ref{cpct-embed}, we obtain
\[
\mathcal{E}_L(u,u) \leq \lambda_1 \quad \text{and} \quad
\into \omega(x) u^2 ~dx = \lim_{n \to \infty} \into \omega(x) u_n^2 ~dx = 1.
\]
This further implies in view of Proposition \ref{simplicity}, either $u = e_1$ or $u = -e_1$.
 Let $u = e_1$. Using Lemma \ref{limit passage} (i) and passing to the limit in $\l|\into \omega(x)(u_n + e_1)^2 ~dx\r|<\rho_n$, we get
\[
4= \l(\into \omega(x) \left(u^2 + e_1^2 \right) ~dx +2\into \omega(x) ue_1 ~dx\r)
= \lim_{n \to \infty} \into \omega(x) (u_n + e_1)^2 ~dx = 0,
\]
which is a contradiction. The same conclusion follows if $u=-e_1$. Therefore, no element of $M$ can be simultaneously close to both $e_1$ and $-e_1$ and thus we have split $A_n$ into two parts. Let
\[
\ph(u) :=
\begin{cases}
1 & \text{if } u \in A_n \cap \mathcal{S}^+_{\rho_n}\\
-1 & \text{if } u \in A_n \cap \mathcal{S}^-_{\rho_n}.
\end{cases}
\]
 Then, $\ph:A_n \to \R\setminus\{0\}$ is a odd, continuous map. Using Definition \ref{genus-def}, we get $\gamma(A_n)\leq 1$,
contradicting $A_n \in F_2$.
\newline
\textbf{Case 2:} Otherwise, assume that there exists a $\rho_0>0$ and a sequence $(u_n)_{n\in \N}\subseteq M$ such that
\[
\mathcal{E}_L(u_n,u_n) \leq \lambda_1 + \frac{1}{n}
\quad 
\text{and} \quad
\l| \into \omega(x)(u_n \pm e_1)^2 dx \r| > \rho_0.
\]
Using the same reasoning as above, we get the existence of $u\in \h(\Omega)$ such that $u_n \rightharpoonup u$ in $\h(\Omega)$ and that $u=\pm e_1$. Taking $u=e_1$ (for $u=-e_1$, the proof follows analogously) and passing to the limit via Lemma \ref{limit passage} (i) in $\into \omega(x) (u_n-e_1) ^2 ~dx$, we get
\[0\geq \rho_0>0,\] a contradiction.
Thus, $\lambda_1<\lambda_2.$
\end{proof}
\begin{proposition}
\label{sec-alter}
Let $\omega\in L^\infty(\Omega)$. The following is true
    \[\lambda_2:=\inf_{A\in F_2} \sup_{u \in A} \mathcal{E}_L(u,u)=\inf_{f\in \mathcal{C}} \max_{\delta\in \mathbb{S}^1} \mathcal{E}_L(f(\delta),f(\delta))=:\tilde{\lambda}_2,\]
    where $\mathcal{C}$ denotes the set of all odd continuous maps from $\mathbb{S}^1$ to $M.$
\end{proposition}

\begin{proof}
Fix $f \in \mathcal{C}$ and set $A=f(\mathbb{S}^1)$. 
By \cite[Example 7.4]{Rabinowitz}, we have $\gamma(A)\geq 2$, hence $A\in F_2$. Moreover, since $\mathbb{S}^1$ is compact and $\mathcal{E}_L(f(\cdot),f(\cdot))$ is continuous, we have
\[
\sup_{u\in A}\mathcal{E}_L(u,u)
=
\sup_{\delta\in\mathbb{S}^1}\mathcal{E}_L(f(\delta), f(\delta))
=
\max_{\delta\in\mathbb{S}^1}\mathcal{E}_L(f(\delta), f(\delta)).
\]
Thus,
\[
\inf_{A\in F_2} \sup_{u \in A} \mathcal{E}_L(u,u)
\leq
\sup_{u\in A}\mathcal{E}_L(u,u)
=
\max_{\delta\in\mathbb{S}^1}\mathcal{E}_L(f(\delta),f(\delta)).
\]
Taking the infimum over all $f\in\mathcal{C}$, we obtain
\begin{equation}
\label{sec-ev-1}
\lambda_2\leq \tilde{\lambda}_2.
\end{equation}
Next, we prove that
\[
\tilde{\lambda}_2 = \min \{ \lambda > \lambda_1 : \lambda \text{ is an eigenvalue of \eqref{lambda}} \}.
\]
By \cite[Proposition 2.7]{Cuesta-1}, $\tilde{\lambda}_2$ is an eigenvalue of \eqref{equ}. Using Proposition \ref{isolate-alter} and \eqref{sec-ev-1}, we get $\lambda_1 < \tilde{\lambda}_2$.
Let $\lambda > \lambda_1$ be an eigenvalue and $u$ be a corresponding eigenfunction.
By Theorem \ref{sign-changing}, we have $u^{\pm} \neq 0$.
Next, we claim that
\begin{equation}\label{sign-claim}
    \mathcal{E}_L(u^+,u^-) \leq 0 
\end{equation}

Using \eqref{sign-ineq} and the definition of $\mathcal{E}_L(\cdot, \cdot)$,  we obtain
\begin{align*}
\begin{split}
\mathcal{E}_L(u,u^+)&\geq  \frac{c_N}{2}{\iint\limits_{\substack{x,y\in\R^N \\ \abs{x-y} \leq 1}}} \frac{|u^+(x)-u^+(y)|^2}{|x-y|^N} ~dx ~dy+\rho_N \int_{\R^N} (u^+)^2 ~dx\\&\qquad\qquad-c_N{\iint\limits_{\substack{x,y\in\R^N \\ \abs{x-y} \geq 1}}} \frac{u^+(x)u^+(y)}{|x-y|^N} ~dx ~dy=\mathcal{E}_L(u^+,u^+)
\end{split}
\end{align*}
which implies the required claim. 

Let $\delta=(\delta_1,\delta_2)\in \mathbb{S}^1$.
Taking $v=u^\pm$ in \eqref{weak-formulatn}, using $u=u^+-u^-$ and the fact that $u^+u^-=0$, we get
\[
\mathcal{E}_L(u^+-u^-,u^+)
= \lambda \into \omega(x) (u^+)^2 ~dx, \quad \mathcal{E}_L(u^+-u^-,u^-)
= -\lambda \into \omega(x) (u^-)^2 ~dx.
\]

In above, multiplying the first relation by $\delta_1^2$, the second by $\delta_2^2$ and subtracting, we get
\[\lambda \left( \delta_1^2 \into \omega(x)(u^+)^2 ~dx
+ \delta_2^2 \into \omega(x)(u^-)^2 ~dx \right)
=\delta_1^2\mathcal{E}_L(u, u^+)- \delta_2^2\mathcal{E}_L(u, u^-).\] Now, using the bilinearity of $\mathcal{E}_L$ and \eqref{sign-claim}, we obtain
\[
\begin{split}
\lambda & \left( \delta_1^2 \into \omega(x)(u^+)^2 ~dx
+ \delta_2^2 \into \omega(x)(u^-)^2 ~dx \right) 
=\delta_1^2\mathcal{E}_L(u, u^+)- \delta_2^2\mathcal{E}_L(u, u^-)\\
& =\delta_1^2\mathcal{E}_L(u^+, u^+) - (\delta_1^2 + \delta_2^2)\mathcal{E}_L(u^-, u^+) + \delta_2^2\mathcal{E}_L(u^-, u^-) \\
& \geq \delta_1^2\mathcal{E}_L(u^+, u^+)+\delta_2^2\mathcal{E}_L(u^-, u^-) - 2 \delta_1 \delta_2\mathcal{E}_L(u^-, u^+) \\
& = \mathcal{E}_L((\delta_1 u^+-\delta_2 u^-), (\delta_1 u^+-\delta_2 u^-)).
\end{split}
\]
 Define $f: \mathbb{S}^1 \to M$ as
\[
f(\delta) = \l(\frac{\delta_1 u^+ - \delta_2 u^-}{\left( \delta_1^2 \into \omega(x)(u^+)^2 ~dx
+ \delta_2^2 \into \omega(x)(u^-)^2 ~dx \right)^{1/2}}\r).
\]
Then, $f \in \mathcal{C}$,
\[
\into \omega(x) (f(\delta))^2 ~dx = 1 \quad \text{and} \quad \mathcal{E}_L(f(\delta),f(\delta)) \leq \lambda.
\]
Taking the maximum over $\mathbb{S}^1$ and infimum over $\mathcal{C}$, we obtain
\[
\tilde{\lambda}_2 \leq \lambda.
\]
Combining this with \eqref{sec-ev-1}, we get the claimed result.
\end{proof}

\subsection{Monotonicity of eigenvalues}
In this subsection, we prove that the eigenvalues obtained in Theorem \ref{exist-eigenval-L} are monotonic {\it w.r.t.} the weight function $\omega$ and the domain $\Omega$.
\begin{proposition}
\label{mono-wei}
Let $\omega_1, \ \omega_2$ be weight functions satisfying \eqref{assump w0}-\eqref{assump w2} such that $\omega_1(x)\leq \omega_2(x)$ a.e. in $\Omega$.
Then for every $k\geq1$, we have
\[
\lambda_k(\omega_2)\leq \lambda_k(\omega_1).
\]
\end{proposition}

\begin{proof}
Denote $M_{\omega}:= \{u \in \h(\Omega): \into \omega(x) u^2 ~dx =1\}.$ Since $\omega_1\leq \omega_2$, then for $u \in M_{\omega_1}$ we have
\[
1 \leq \into \omega_2(x) u^2 ~dx.
\]
For $u \in M_{\omega_1}$, define
\[
v=\frac{u}{\l(\into \omega_2(x) u^2 ~dx\r)^\frac{1}{2}} \quad \text{such that} \ v \in M_{\omega_2}
\]
Moreover,
\[
\mathcal{E}_L(v,v)=\frac{\mathcal{E}_L(u,u)}{\into \omega_2(x) u^2~dx}\leq \mathcal{E}_L(u,u).
\]
Let $A\subset \{u\in \h(\Omega): \into \omega_1(x) u^2 ~dx=1\}$ be a closed, symmetric set with genus $\gamma(A)\geq k$.
Define
\[
B=\l\{\frac{u}{\l(\into \omega_2(x) u^2~dx\r)^\frac{1}{2}}:u\in A\r\}.
\]
Note that since $A$ is symmetric, $B$ is also symmetric. It is easy to deduce that $B\subseteq A$ and $B\subseteq \{u\in \h(\Omega): \into \omega_2(x) u^2 ~dx=1\}$ is closed. Consider the map $T:A\to \h(\Omega)$ defined as \[T(u)=\frac{u}{\sqrt{\into \omega_2(x)u^2~dx}}.\] It is easy to see that $T$ is odd, continuous and $T(A)=B.$ By Lemma \ref{genus-prop} (ii), we get $\gamma(T(A))=\gamma(B)\geq \gamma(A)\geq k$.
Hence
\[
\sup_{v\in B}\mathcal{E}_L(v,v)\leq \sup_{u\in A}\mathcal{E}_L(u,u).
\]
Taking the infimum over all such sets, we obtain $\lambda_k(\omega_2)\leq \lambda_k(\omega_1).$
\end{proof}

\begin{proposition}
Let $\Omega_1 \subsetneq \Omega_2\subseteq \R^N$ be bounded domains. Then, it holds that
\[
\lambda_k(\Omega_2) \leq \lambda_k(\Omega_1)
\quad \text{for all } k\geq 1 .
\]
\end{proposition}

\begin{proof}
Note that for each $k \in \N$, the admissible sets satisfies
\[F_{k, \Omega_2} \subset F_{k, \Omega_1}\]
where $F_{k, \Omega}$ is defined as 
\[F_{k, \Omega}:=\{A\subseteq M_\Omega \,:\, A=-A \ \text{is closed }, \ \gamma(A)\geq k\}\]
and \[
M_{\Omega}:=\{u \in \h(\Omega)\,:\, J(u)=1\}.
\]
Finally, the proof follows from the variational characterization of eigenvalues $\lambda_k$ defined in \eqref{lambda}.
\end{proof}

In contrast, the previous argument cannot, in general, be extended to establish strict decreasing monotonicity, due to the lack of compactness of the sets $A \in F_k$. However, using the alternative characterizations of $\lambda_1$ and $\lambda_2$ established in Theorem \ref{existence lambda1} and Proposition \ref{sec-alter} respectively, we derive the strict monotonicity properties of the first and second eigenvalue $\lambda_1$ and $\lambda_2$ {\it w.r.t.} the bounded weight function $\omega$ and the domain $\Omega$.

\begin{proposition}
Let $\omega_1, \ \omega_2$ be bounded weight functions such that $\omega_1 \leq \omega_2$ in $\Omega$, $\omega_1 \neq \omega_2$. Then,
\[
\lambda_1(\omega_1) > \lambda_1(\omega_2).
\]
\end{proposition}

\begin{proof}
Let $e_{1,\omega_1}$ be the eigenfunction corresponding to $\lambda_1(\omega_1).$ By Proposition \ref{constant sign e-1} and \cite[Theorem 1.1]{Jarohs-Weth}, $e_{1,\omega_1}>0$. Since $\omega_2>\omega_1$ on a subset of $\Omega$ with positive measure and $e_{1,\omega_1} > 0$ in $\Omega$, we have
\[
\into \omega_2(x) e_{1,\omega_1}^2 ~dx > \into \omega_1(x) e_{1,\omega_1}^2 ~dx = 1.
\]
Setting
\[
v = \frac{e_{1,\omega_1}}{\left(\into \omega_2(x)e_{1,\omega_1}^2 ~dx\right)^{1/2}}.
\]
Then, $\into \omega_2(x) v^2=1$. By \eqref{lambda}, we have
\[
\lambda_1(\omega_2) \leq \mathcal{E}_L(v,v)
= \frac{\mathcal{E}_L(e_{1,\omega_1},e_{1,\omega_1})}{\into \omega_2(x) e_{1,\omega_1}^2 ~dx}<\frac{\mathcal{E}_L(e_{1,\omega_1},e_{1,\omega_1})}{\into \omega_1(x) e_{1,\omega_1}^2 ~dx}=\mathcal{E}_L(e_{1,\omega_1},e_{1,\omega_1})
= \lambda_1(\omega_1),
\]
which concludes the proof.
\end{proof}

\begin{proposition}
Let $\omega_1,\omega_2$ be bounded weight functions such that $\omega_1 < \omega_2$ in $\Omega$. Then,
\[
\lambda_2(\omega_1) > \lambda_2(\omega_2).
\]
\end{proposition}

\begin{proof}
For any $\lambda > \lambda_1(\omega_1)$, we claim that
\begin{equation}
\label{cgt1}
m_{\omega_1,\omega_2,\lambda} > 1,
\end{equation} where
\[
m_{\omega_1,\omega_2,\lambda} = \inf \left\{ \int_\Omega \omega_2(x) u^2 ~dx : \into \omega_1(x) u^2 ~dx=1,\ \mathcal{E}_L(u,u) \leq \lambda \right\}.
\]
Note that $m_{\omega_1,\omega_2,\lambda}$ is well defined due to \eqref{lambda}.
Since $\omega_1 < \omega_2$ in $\Omega$, we have $m_{\omega_1,\omega_2,\lambda} \geq 1$. Let $(u_n)_{n \in \N}\subseteq \{u \in \h(\Omega):\into \omega_1(x) u^2 ~dx=1\}$ be a minimizing sequence for $m_{\omega_1,\omega_2,\lambda}$. Then, by using \eqref{assump w1} and \eqref{E-L-bd-below-above}, we get $(u_n)_{n \in \N}$ is bounded in $\h(\Omega)$. Thus,
\[
u_n \rightharpoonup u \quad \text{in } \h(\Omega), \text{ upto a subsequence}.
\]
Moreover, by using the weak lower semicontinuity of $\h(\Omega)$ norm, Lemma \ref{limit passage} (i), \eqref{est:lower:conv} and Theorem \ref{cpct-embed}, we have
\[
\mathcal{E}_L(u,u) \leq \liminf_{n \to \infty} \mathcal{E}_L(u_n,u_n) \leq \lambda,
\]
\[
\into \omega_1(x) u^2 ~dx = \lim_{n \to \infty} \into \omega_1(x) u_n^2 ~dx = 1,\]
and
\[\into \omega_2(x) u^2 ~dx = \lim_{n \to \infty} \into \omega_2(x) u_n^2 ~dx  = m_{\omega_1,\omega_2,\lambda}.
\]
Using this and $\omega_1 < \omega_2$ in $\Omega$, we get
\[
m_{\omega_1,\omega_2,\lambda}
= \into \omega_2(x) u^2 ~dx
> \into \omega_1(x) u^2 ~dx = 1,
\]
which proves \eqref{cgt1}.
Now, by Proposition \ref{mono-wei}, we have $\lambda_2(\omega_2) \leq \lambda_2(\omega_1)$. By Proposition \ref{sec-alter} there exists $(f_n)_{n \in \N} \subset \mathcal{C'},$ the set of all odd continuous maps from $\mathbb{S}^1$ to $\{u \in \h(\Omega): \into \omega_1(x) u^2 ~dx=1\}$ such that
\[
\max_{\delta\in \mathbb{S}^1} \mathcal{E}_L(f_n(\delta),f_n(\delta)) \leq \lambda_2(\omega_1) + \frac{1}{n}.
\]
Since $\omega_1 < \omega_2$ in $\Omega$, we have
\[
\int_\Omega \omega_2(x)(f_n(\delta))^2 ~dx > \int_\Omega \omega_1(x)(f_n(\delta))^2~dx = 1 \quad \text{for every } \delta \in \mathbb{S}^1.
\]
For every $n \in \N$, $\delta \in \mathbb{S}^1$, let
\[
\tilde f_n(\delta) := \frac{f_n(\delta)}{\left(\int_\Omega \omega_2(x)(f_n(\delta))^2~dx\right)^{1/2}}.
\]
Then, $\into \omega_2(x) (f_n(\delta))^2 ~dx=1.$
It is easy to see that $\tilde f_n \in \mathcal{C''},$ the set of all odd continuous maps from $\mathbb{S}^1$ to $\{u \in \h(\Omega): \into \omega_2(x) u^2 ~dx=1\}.$ Since $\mathbb{S}^1$ is compact, for each $n \in \N$, we can find $\delta_n \in \mathbb{S}^1$ such that
\[
\mathcal{E}_L(\tilde f_n(\delta_n), \tilde f_n(\delta_n)) = \max_{\delta\in \mathbb{S}^1} \mathcal{E}_L(\tilde f_n(\delta), \tilde f_n(\delta)).
\]
Using $\lambda_2(\omega_1)+1 > \lambda_1(\omega_1)$, by Proposition \ref{sec-alter} and the above estimate, we have
\[
\lambda_2(\omega_2)
\leq \mathcal{E}_L(\tilde f_n(\delta_n), \tilde f_n(\delta_n))
= \frac{\mathcal{E}_L( f_n(\delta_n), f_n(\delta_n))}{\int_\Omega \omega_2(x) (f_n(\delta_n))^2 ~dx}
\leq \frac{\lambda_2(\omega_1) + \frac{1}{n}}{m_{\omega_1,\omega_2,\lambda_2(\omega_1)+1}}.
\]
Passing to the limit as $n\to\infty$ and using \eqref{cgt1}, we get
\[
\lambda_2(\omega_2)
\leq \frac{\lambda_2(\omega_1)}{m_{\omega_1,\omega_2,\lambda_2(\omega_1)+1}} < \lambda_2(\omega_1),
\]
we get the claimed result.
\end{proof}

\begin{proposition}
Let $\omega\in L^\infty(\Omega)$ and  $\Omega_1, \ \Omega_2$ be bounded domains in $\R^N$ such that $\Omega_1$ is a proper subset of $\Omega_2$. Then,
\[
\lambda_1(\Omega_2) < \lambda_1(\Omega_1).
\]
\end{proposition}

\begin{proof}
Let $u \in \h(\Omega_1)$ be a positive eigenfunction associated to $\lambda_1(\Omega_1).$

Note that $v=\frac{u}{\sqrt{\int_{\Omega_2} \omega(x) \tilde{u}^2 ~dx}} \in \h(\Omega_2)$ and using \eqref{lambda}, we obtain
\[
\lambda_1(\Omega_2)\leq \mathcal{E}_L\l(v,v\r) =
\frac{\mathcal{E}_L(u, u)}
{\int_{\Omega_2} \omega(x) u^2  ~dx}
=
\frac{\mathcal{E}_L(u,u)}
{\int_{\Omega_1} \omega(x) u^2  ~dx}
= \lambda_1(\Omega_1).
\]
The equality holds only if $u$ is an eigenfunction associated to $\lambda_1(\Omega_2)$. This is not possible because
$
|\{ x \in \Omega_2 : \tilde{u}(x)=0 \}| > 0,
$
which contradicts the fact that the first eigenfunction is strictly positive in $\Omega$ (in view of Proposition \ref{constant sign e-1} and the strong maximum principle \cite[Theorem 1.1]{Jarohs-Weth}). Hence, the required claim. 

\end{proof}

\appendix
\section{}
\renewcommand{\thelemma}{\thesection.\arabic{lemma}}
\renewcommand{\thetheorem}{\thesection.\arabic{theorem}}
\label{Appendix}
In this section, we establish a Hardy-type inequality for the operator $L_\Delta$, as stated in Theorem \ref{Hardy-inequ}. The proof is based on differentiating the following Hardy inequality (see \cite[inequality (1.2)]{Dipierro}) with respect to the parameter $s$ at $s=0$. For any $u \in C_c^\infty(\R^N)$, we have
\begin{align}
\label{hardy-frac}
    \begin{split}
        \frac{2^{2s}\Gamma^2\l(\frac{N+2s}{4}\r)}{\Gamma^2\l(\frac{N-2s}{4}\r)} \int_{\R^N} \frac{u^2}{|x|^{2s}} ~dx\leq c_{N,s}\intr \frac{|u(x)-u(y)|^2}{|x-y|^{N+2s}} ~dx ~dy,
    \end{split}
    \end{align}
    where $c_{N,s}:=\frac{2^{2s}\pi^{\frac{-N}{2}}s \Gamma(\frac{N+2s}{2})}{\Gamma(1-s)}$.
\newline
\textbf{Proof of Theorem \ref{Hardy-inequ}:}
Let $u \in C_c^\infty(\R^N)\setminus\{0\}.$ Note that,
    \[
    \begin{split}
    \frac{d}{ds}\l(\frac{2^{2s}\Gamma^2\l(\frac{N+2s}{4}\r)}{\Gamma^2\l(\frac{N-2s}{4}\r)} \int_{\R^N} \frac{u^2}{|x|^{2s}} ~dx\r)&=\frac{2^{2s}\Gamma^2\l(\frac{N+2s}{4}\r)}{\Gamma^2\l(\frac{N-2s}{4}\r)}\bigg(2\ln2+\Psi(\frac{N+2s}{4})\\&\qquad\qquad\qquad\qquad+\Psi(\frac{N-2s}{4})\bigg) \int_{\R^N} \frac{u^2}{|x|^{2s}} ~dx\\&\qquad\qquad-\frac{2^{2s+1}\Gamma^2\l(\frac{N+2s}{4}\r)}{\Gamma^2\l(\frac{N-2s}{4}\r)}\int_{\R^N} \frac{u^2\ln|x|}{|x|^{2s}} ~dx.
    \end{split}
    \]
    Applying Taylor's expansion about $s=0$, we get
    \begin{align}
    \label{frac-exp}
    \begin{split}
    \frac{2^{2s}\Gamma^2\l(\frac{N+2s}{4}\r)}{\Gamma^2\l(\frac{N-2s}{4}\r)} \int_{\R^N} \frac{u^2}{|x|^{2s}} ~dx&=\int_{\R^N} u^2 ~dx+s\bigg[2\bigg(\ln2+\Psi(\frac{N}{4})\bigg) \int_{\R^N} u^2 ~dx\\&\qquad\qquad\qquad-2\int_{\R^N} u^2\ln|x| ~dx\bigg]+o(s).
    \end{split}
    \end{align}
Using {Lemma \cite[Lemma 3.7]{Santamaria-Saldana}, we get}
    \begin{equation}
    \label{E_L-exp}
    c_{N,s}\intr \frac{|u(x)-u(y)|^2}{|x-y|^{N+2s}} ~dx ~dy=\int_{\R^N} u^2 ~dx +s\mathcal{E}_L(u,u)+o(s).
    \end{equation}
Substituting  \eqref{frac-exp} and \eqref{E_L-exp} in \eqref{hardy-frac}, we have
     \[
     \begin{split}
        &\int_{\R^N} u^2 ~dx+s\bigg[2\bigg(\ln2+\Psi(\frac{N}{4})\bigg) \int_{\R^N} u^2 ~dx-2\int_{\R^N} u^2\ln|x| ~dx\bigg]+o(s) \\
        &\qquad\qquad\qquad\qquad\leq \int_{\R^N} u^2 ~dx +s\mathcal{E}_L(u,u)+o(s).
     \end{split}
     \]
     Thus,
     \[2\bigg(\ln2+\Psi(\frac{N}{4})\bigg) \int_{\R^N} u^2 ~dx-2\int_{\R^N} u^2\ln|x| ~dx \leq \mathcal{E}_L(u,u).\]Hence, the claim follows for all non zero compactly supported functions in $\R^N.$ The inequality for every $u \in \h(\Omega)$ follows by density arguments (\cite[Theorem 3.1]{Chen-Weth}) and the fact that $\Omega$ is bounded.

\qed

\section*{Acknowledgment}  
The first author acknowledges the financial support from the Anusandhan National Research Foundation (ANRF), India, under Grant No. ANRF/ARGM/2025/ 000272/MTR.
The third author is supported by the UGC Senior Research Fellowship (Reference No. 221610015405).

\end{document}